\documentclass[journal,twoside,web]{ieeecolor}

\usepackage{generic}
\usepackage{cite}
\usepackage{amsmath,amssymb,amsfonts}
\usepackage{algorithmic}
\usepackage{graphicx}
\usepackage{textcomp}
\usepackage{bbm, xspace}
\usepackage{mathrsfs}
\usepackage{mathtools}
\usepackage{color}
\usepackage{epsfig}
\usepackage[acronym]{glossaries}
\usepackage{cancel}
\usepackage{empheq}
\let\labelindent\relax
\usepackage{enumitem}
\usepackage{textgreek}
\usepackage{stackengine}
\usepackage{dblfloatfix}
\usepackage[font=small]{caption}
\usepackage{empheq}
\usepackage{flushend}
\usepackage{balance}
\usepackage{subcaption}

\newcommand{\ie}{i.e.,\@\xspace} 
\newcommand{\eg}{e.g.,\@\xspace} 

\def\be{\begin{equation}}
	\def\ee{\end{equation}}

\newcommand{\bs}{\boldsymbol}
\newcommand{\mc}{\mathcal}

\definecolor{darkgreen}{rgb}{0.0, 0.5, 0.0}

\renewcommand{\emph}{\textit}
\newcommand\numberthis{\addtocounter{equation}{1}\tag{\theequation}}

\makeglossaries
\newacronym{GNEP}{GNEP}{generalized Nash equilibrium problem}
\newacronym{NE}{NE}{Nash equilibrium}
\newacronym{NEP}{NEP}{Nash equilibrium problem}
\newacronym{GNE}{GNE}{generalized Nash equilibrium}
\newacronym{GNEshort}{GNE}{generalized \gls{NE}}
\newacronym{v-GNE}{v-GNE}{variational \gls{GNE}}
\newacronym{ISS}{ISS}{Input-to-state-stable}
\newacronym{PPA}{PPA}{proximal-point algorithm}
\newacronym{PPPA}{PPPA}{preconditioned \gls{PPA}}
\newacronym{VI}{VI}{variational inequality}
\newacronym{GAE}{GAE}{generalized aggregative equilibrium}
\newacronym{v-GAE}{v-GAE}{variational \gls{GAE}}
\newacronym{KKT}{KKT}{Karush--Kuhn--Tucker}
\newacronym{FQNE}{FQNE}{firmly quasinonexpansive}
\newacronym{FNE}{FNE}{firmly nonexpansive}
\newacronym{ADMM}{ADMM}{alternating direction method of multipliers}
\newacronym{MPMM}{MPMM}{modified proximal method of multipliers}
\newacronym{OPF}{OPF}{optimal power flow}
\newacronym{OPFP}{OPFP}{optimal power flow problem}
\newacronym{NUM}{NUM}{network utility maximization}
\newacronym{END}{END}{Estimation Network Design}
\newacronym{DGD}{DGD}{distributed gradient descent}
\newacronym{ST}{ST}{Steiner Tree}
\newacronym{FB}{FB}{forward-backward}

\newcommand{\0}{\bs 0}
\def\1{{\bs 1}}

\def\argmin{\mathop{\rm argmin}}

\newcommand{\minimize}[1]{\displaystyle\minim_{#1}}
\newcommand{\minim}{\mathop{\hbox{\rm minimize}}}

\newcommand{\col}{\mathrm{col}}

\def\diag{\mathop{\hbox{\rm diag}}}

\def\Null{\mathop{\hbox{\rm null}}}

\def\range{\mathop{\hbox{\rm range}}}

\def\grad#1{{\nabla_{\! #1}}}

\def\spose#1{\hbox to 0pt{#1\hss}}

\newcommand{\proj}{\operatorname{proj}}
\newcommand{\prox}{\operatorname{prox}}
\def\fix{\mathrm{fix}}
\def\gra{\operatorname{gra}}
\def\zer{\operatorname{zer}}
\def\dom{\operatorname{dom}}

\newcommand{\Id}{\op{Id}}
\def\nc{\op{N}}
\def\Hmc{{\mathcal{H}}}

\def\R{\mathbb{R}}
\def\nat{\mathbb{N}}

\def\X{\mathcal{X}}

\def\ybs{{\bs y}}

\DeclareSymbolFont{myletters}{OML}{ztmcm}{m}{it}
\DeclareMathSymbol{\uplambda}{\mathord}{myletters}{"15}

\def\QEDhereeqn{\eqno\let\eqno\relax\let\leqno\relax\let\veqno\relax\hbox{\QED}}
\def\QEDopenhereeqn{\eqno\let\eqno\relax\let\leqno\relax\let\veqno\relax\hbox{\QEDopen}}

\makeatletter \let\cl@part\relax \makeatother
\usepackage{nohyperref}
\usepackage[capitalize]{cleveref}

\def\k{{k \in \nat}}  

\crefname{thm}{Theorem}{Theorems}
\crefname{lem}{Lemma}{Lemmas}
\crefname{cor}{Corollary}{Corollaries}
\crefname{claim}{Claim}{Claims}
\crefname{axiom}{Axiom}{Axioms}
\crefname{conj}{Conjecture}{Conjectures}
\crefname{fact}{Fact}{Facts}
\crefname{hypo}{Hypothesis}{Hypotheses}
\crefname{assum}{Assumption}{Assumptions}
\crefname{prop}{Proposition}{Propositions}
\crefname{crit}{Criterion}{Criteria}
\crefname{standing}{Standing Assumption}{Standing Assumptions}
\crefname{defn}{Definition}{Definitions}
\crefname{exmp}{Example}{Examples}
\crefname{rem}{Remark}{Remarks}
\crefname{prob}{Problem}{Problems}
\crefname{prin}{Principle}{Principles}
\crefname{alg}{Algorithm}{Algorithms}
\crefname{figure}{Figure}{Figures}
\crefname{assumption}{Assumption}{Assumptions}

\crefname{thmlisti}{Theorem}{Theorems}
\crefname{lemlisti}{Lemma}{Lemma}
\crefname{asmlisti}{Assumption}{Assumption}

\newtheorem{theorem}{Theorem}

\newtheorem{lemma}{Lemma}

\newtheorem{assumption}{Assumption}

\newtheorem{standing}{Standing Assumption}
\newtheorem{definition}{Definition}
\newtheorem{example}{Example}
\newtheorem{remark}{Remark}

\newlist{thmlist}{enumerate}{1}
\setlist[thmlist]{label={\it(\roman{thmlisti})}, ref=\thepb{(\it \roman{thmlisti})},noitemsep,topsep=0em,leftmargin=*}

\newlist{thmlist2}{enumerate}{1}
\setlist[thmlist2]{label={(\alph{thmlist2i})}, ref={{(\alph{thmlist2i})}},topsep=0em}

\newlist{asmlist}{enumerate}{1}
\setlist[asmlist]{label={\it(\roman{asmlisti})}, ref=\theassumption{(\it \roman{asmlisti})},noitemsep,topsep=0em,leftmargin=*}

\usepackage{thmtools}
\declaretheorem[
name=Design Problem,
Refname={Design Problem,Design Problem},
]{pb}
\Crefname{pb}{Problem}{Problems}
\addtotheorempostheadhook[pb]{\crefalias{thmlisti}{pb}}

\Crefname{asm}{Assumption}{Assumptions}
\addtotheorempostheadhook[assumption]{\crefalias{thmlisti}{assumption}}

\def\I{\mc{I}}
\def\i{i\in\I}

\def\M{\mc{M}}
\def\m{{m\in\mc{M}}}
\def\P{\mc{P}}
\def\p{{p\in\mc{P}}}
\def\q{{q\in\mc{Q}}}
\def\Q{\mc{Q}}

\def\y{y}
\def\hy{\h{y}}
\def\hyt{\tilde{\h{y}}}

\def\hyhat{\hat{\h{y}}}

\def\hypar{{\h{y}}_{\scriptscriptstyle \parallel}}
\def\hyperp{{\h{y}}_{\scriptscriptstyle \perp}}
\def\hyhatpar{{\h{y}}_{\scriptscriptstyle \parallel}}
\def\hyhatperp{\hat{\h{y}}_{\scriptscriptstyle \perp}}
\def\yhatpar{{{y}}_{\scriptscriptstyle \parallel}}
\def\hlambda{\h{\lambda} }
\def\hsigma{\h{\sigma}}
\def\hsigmat{\tilde{\h{\sigma}}}
\def\hs {\h{s}}
\def\hz{\h{z}}

\def\hystar{\h{y}^{\star}}

\def\mcV {\mc{V}}

\def\homega{\h{\omega}}

\def\eigmin{\uplambda_{\operatorname{min}}}
\def\eigmax{\uplambda_{\operatorname{max}}}

\def \Fbs{{   { \bs{\op F}}   }}
\def\Fbstilde{ { \bar {\bs{ \op F}}}}
\def \Rmc{{   \textrm{R}   }}
\def \Fmc{{   \mc{F}   }}

\def\n#1{{n_{#1}}}

\def\Ebs#1{{\h{\mc{C}}}_{#1}}
\def\Ebsperp#1{{\h{\mc{C}}}{}_{\!\scriptscriptstyle \perp}^{  \mat{#1}}}

\def\Etbs#1{\tilde{\h{\mc{C}}}_{#1}}

\def\perm{\textrm{P}}
\def\Piparallel#1{\Pi_{\scriptscriptstyle \parallel}^{\mat{#1}}}
\def\Piperp#1{\Pi_{\!\scriptscriptstyle \perp}^{\mat{#1}}}

\def\op{\mathsf}
\def\mat{\mathrm}
\def\set{\mc}

\def\h{\bs}
\def\design{design}
\def\id{\mat{I}}
\def\Ntilde#1{ { N}_{#1}}

\def\E#1{\mc{E}^{\scriptstyle \text{#1}}}
\def\g#1{\mc{G}^{\scriptstyle \text{#1}}}
\def\V#1{\mc{V}^{\scriptstyle \text{#1}}}
\def\W#1{\mat W^{\scriptstyle \text{#1}}}
\def\L#1{\mat L^{\scriptstyle \text{#1}}}
\def\D#1{\mat D^{\scriptstyle \text{#1}}}
\def\w#1{w^{\scriptstyle \text{#1}}}

\def\neig#1#2{\mc{N}{~\!\!}^{\scriptstyle \text{#1}}({#2})}
\def\neigo#1#2{\overline{\mc{N}}{~\!\!}^{\scriptstyle \text{#1}}({#2})}

\def\gD#1{{\mc{G}}^{\scriptstyle \text{D}}_{#1}}

\def\WD#1{{\mat W}^{\scriptstyle \text{D}}_{#1}}
\def\LD#1{{ \mat L}^{\scriptstyle \text{D}}_{#1}}
\def\ED#1{\mc{E}^{\scriptstyle \text{D}}_{#1}}
\def\neigD#1#2{{\mc{N}}{~\!\!}^{\scriptstyle \text{D}}_{#1}(#2)}

\def\qD#1{{q}^{\scriptstyle \text{D}}_{#1}}

\def\WDbs{ {\h{\mat W}}^{\scriptstyle \text{D}}}
\def\WDtbs{\tilde{\bs{ \mat W}}{}^{\scriptstyle \text{D}}}
\def\LDbs{\h{\mat L}^{\scriptstyle \text{D}}}
\def\LDtbs{\tilde{\bs{\mat L}}{}^{\scriptstyle \text{D}}}

\def\Esup#1#2{\mc{E}^{\scriptstyle \text{#1},#2}}
\def\gsup#1#2{\mc{G}^{\scriptstyle \text{#1},#2}}

\def\neigsup#1#2#3{\mc{N}{~\!\!}^{\scriptstyle \text{#1},#3}({#2})}
\def\neigosup#1#2#3{\overline{\mc{N}}{~\!\!}^{\scriptstyle \text{#1},#3}({#2})}

\def\gDsup#1#2{{\mc{G}}^{\scriptstyle \text{D},#2}_{#1}}

\def\WDsup#1#2{{\mat W}^{\scriptstyle \text{D},#2}_{#1}}
\def\LDsup#1#2{{ \mat L}^{\scriptstyle \text{D},#2}_{#1}}

\def\LDbssup#1{\h{\mat L}^{\scriptstyle \text{D},#1}}
\def\LDsup#1{{\mat L}^{\scriptstyle \text{D},#1}}

\def\gDlambda#1{{\mc{G}}^{\scriptstyle \text{D},\lambda}_{#1}}
\def\gDsigma#1{{\mc{G}}^{\scriptstyle \text{D},\sigma}_{#1}}

\begin{document}

\title{
	The END:
	Estimation Network Design for games under partial-decision information
}

\author{ Mattia Bianchi, \IEEEmembership{Member, IEEE}, Sergio Grammatico, \IEEEmembership{Senior Member, IEEE}
	\thanks{Mattia Bianchi is with the Automatic Control Laboratory, ETH Zürich, Switzerland (\texttt{mbianch@ethz.ch}). Sergio Grammatico is with the Delft Center for Systems and Control, TU Delft, The Netherlands (\texttt{s.grammatico@tudelft.nl}). This work was  supported by NWO (OMEGA 613.001.702) and by the ERC (COSMOS, 802348).}
}

\maketitle

\begin{abstract}
    Multi-agent decision problems are typically solved via distributed iterative algorithms, where the agents only communicate between themselves on a peer-to-peer network. 
     Each agent usually maintains a copy of each decision variable, while agreement among the local copies is enforced via consensus protocols. Yet, each agent is often directly influenced  by a small portion of the decision variables only: neglecting this sparsity results in redundancy, poor scalability with the network size, communication and memory overhead. 
     To address these challenges, we develop \gls{END}, a framework for the  design and analysis of distributed algorithms, generalizing  several recent approaches. \gls{END} algorithms can  be tuned to  exploit problem-specific sparsity structures, by optimally allocating  copies of each variable  only to a subset of agents, to improve efficiency and minimize redundancy.
     We illustrate the  \gls{END}'s potential by designing new algorithms for \gls{GNE} seeking under partial-decision information, that can leverage the sparsity in cost functions, constraints and aggregation values. 
     Finally, we test numerically our methods on a unicast rate allocation problem, revealing greatly reduced communication and memory costs. 
     
\end{abstract}

\section{Introduction}\label{sec:introduction}

\IEEEPARstart{L}{arge-scale} problems in machine learning \cite{Boyd_ADMM_2010}, signal processing \cite{FacchineiKanzow_Generalized_4OR2007} and decentralized control \cite{Durr_Stankovic_Johansson_ACC_2011}
 involve huge volumes of data, often spatially scattered.  In these settings, \emph{distributed} multi-agent computation (where the computational burden is partitioned among a group of agents, only allowed to share information over a peer-to-peer communication network, without  central  gathering of the data) is emerging as a fundamental paradigm to enable scalability, privacy preservation and robustness. On the downside, distributed algorithms require storage and transmission of  multiple local copies of some variables, a form of redundancy  that is absent in centralized processing. 
 This may result in prohibitive  memory and communication requirements, and hinders scalability  when  the  dimension of the copies grows with the network size. 

For example, in \gls{NE} problems under \emph{partial-decision information} \cite{YeHu_Consensus_TAC2017,Romano_Penalty_TCNS2023}, each of a group of agents $i \in \mc{I} \coloneqq \{1,2,\dots,N\}$
 has a private cost function $f_i(x_i,x_{-i})$, 
 depending both on its action (decision variable) $x_i$ and on the actions of the other agents $x_{-i} = (x_j)_{j \neq i} $. The goal is to compute a \gls{NE}, namely a set of actions that simultaneously solves the coupled problems 
 \begin{align} \label{eq:gameintro}
    \minimize{x_i}  \ f_i(x_i,x_{-i}), \qquad \forall \i. 
\end{align}
The complication is that each agent can only receive data from some neighbors over a communication graph, although its cost depends on the actions of \emph{non-neighboring} agents. As agent $i$ only has ``\emph{partial information}" about $x_{-i}$, standard \gls{NE} refinement procedures, such as the pseudo-gradient method
\begin{align}
    x_i^{k+1} = x_i^k - \nabla_{x_i} f_i (x_i^k,x_{-i}^k) \qquad \i,  \k,
\end{align}
cannot be implemented distributedly. Instead, to cope with the lack of knowledge, it is typically assumed that each agent keeps and exchanges at each iteration with its neighbors an estimate of  the action of every other agent \cite{YeHu_Consensus_TAC2017,Pavel_GNE_TAC2020,Bianchi_GNEPPP_AUT2022,Gadjov_Resilient_TCNS2023,Zhu_GNEdigraphs_TCNS2021}: yet, this might be impractical, especially if the number of agents $N$ is large. On the other hand, the cost of agent $i$ often  depends only on the actions of  much smaller subset of agents $\neig{I}{i} \subset \mc{I}$, possibly of cardinality independent of $N$ \cite{Salehisadaghiani_Graphical_AUT2018}, e.g. in congestion control problems for wireless \cite{Alpcan:Basar:CongestionControl:CDC:2012} and optical \cite{Pavel_optical_2012} networks. To highlight this dependency, let us  write
 \begin{align}\label{eq:partiallycoupledcost}
     f_i(x_i,x_{-i}) = f((x_j)_{j\in \neig{I}{i}}).
 \end{align}

Unfortunately, the available \gls{NE} seeking algorithms cannot take advantage of the \emph{sparsity} structure in \eqref{eq:partiallycoupledcost}. The sole exception are the algorithms proposed by Salehisadaghiani and Pavel  \cite{Salehisadaghiani_Graphical_AUT2018,Salehisadaghiani_Nondoubly_EAI2020}, where each agent is only assigned  proxies of the decisions that directly influence its cost -- but provided that the communication network can be  freely designed, which is for example not the case for ad hoc networks. 

Similarly, in \gls{GNEshort} seeking (where the agents are also coupled via shared constraints), each agent is typically required
to keep an estimate of all dual variables \cite{Pavel_GNE_TAC2020,Bianchi_GNEPPP_AUT2022}, while possibly being directly affected by only a few of them. Methods to improve memory and communication efficiency in this setup are unknown in the literature.

Such scenarios are not limited to games. in fact, many big-data \emph{optimization} problems are also \emph{partially separable} \cite{NecoaraClipici_Coordinate_SIAM2016} (\ie  the local cost of each agent only depends on a limited portion of the optimization vector)
-- but this fact is rarely exploited in \emph{distributed} algorithms. 
The issue was solely considered, via dual methods, by  Mota et al. \cite{Mota:LocalDomains:TAC:2014}, and
later by Alghunaim, Yuan and Sayed \cite{Alghunaim_SparseConstraints_TAC2020,Alghunaim_Stochastic_TAC2020}: 
in their approach, each component of the optimization variable  is only estimated by a suitable subset of the agents. As a drawback, the dual reformulation is  effective over undirected communication networks and viable for optimization problems only.

\smallskip
\emph{Motivation:}
Our work is motivated by the observation that, in  multi-agent  applications,  the coupling among the agents often  exhibits some sparsity. This sparsity \emph{could} and \emph{should} be exploited to enhance efficiency of distributed algorithms, e.g. by reducing the number of redundant estimates in the network.
While ad hoc schemes were developed for  special sparsity patterns (e.g. when each agent is only coupled with its communication neighbors \cite{Notarnicola_Partitioned_TCNS2018} or neighbors' neighbors \cite{Salehisadaghiani_Graphical_AUT2018}), what is missing is a systematic methodology to account for general, problem-specific, sparsity structures -- but without resorting to a case-by-case convergence analysis.

\emph{Contributions:} To fill this gap, we introduce  \glsfirst{END}, 
a rigorous graph-theoretic language to describe how the estimates of any \emph{variables of interest} are allocated and combined among the agents in any distributed algorithm (Section~\ref{sec:mathbackground}).
The \gls{END} notation has the following advantages:
\begin{itemize}[leftmargin=*]
    \item 
\emph{Versatile and algorithm-free}: 
	the variables of interest include  any quantity  some agents need to reach consensus upon (decision vectors or dual multipliers in optimization problems, the gradient of a cost,  an aggregative value), making \gls{END} applicable to  virtually any networked decision problem, including  consensus optimization \cite{Nedic_DIGing_SIAM2017}, game equilibrium seeking on networks,  common fixed point computation \cite{FullmerMorse_Paracontractions_TAC2018};
 \item \emph{High customizability yet unified analysis}: the estimates allocation in \gls{END} algorithms is a degree of freedom, and can be tailored to specific problem instances, for example by embedding efficiency criteria (e.g., minimal memory allocation, bandwidth constraints;  see Section~\ref{subsec:design}). This is without requiring case-by-case analysis: in fact,  \gls{END} unifies the convergence proof of  sparsity-unaware algorithms (e.g. \cite{Salehisadaghiani:Pavel:gossip:AUT:2016}) with that of algorithms specifically devised for problems with unique sparsity structure (e.g.\cite{Salehisadaghiani_Graphical_AUT2018}).
\end{itemize}


Specifically, in this paper, we showcase the \gls{END} framework on the topical problem of  \gls{GNE} seeking under partial-decision information. We analyze two \gls{END} algorithms and discuss  how they improve the literature, in terms of both theoretical guarantees and communication/memory efficiency. Concretely:
\begin{enumerate}[leftmargin=*]
	\item
	For \gls{NE} problems, we prove linear convergence of a  pseudo-gradient algorithm over directed graphs. By leveraging the expressivity of the \gls{END} framework, we are able to relax the assumptions on the communication network  postulated in the existing literature \cite{Bianchi_Directed_CDC2020,TatarenkoShiNedic_Geometric_TAC2021} (specifically, we admit row stochastic graphs and do not rely on  Perron eigenvectors).  Further, we allow for much more ductile estimate assignment compared to\cite{Salehisadaghiani_Graphical_AUT2018,Salehisadaghiani_Nondoubly_EAI2020}. In particular, special cases of our method recover both the full- and the partial-decision information setup (and a plethora of intermediate scenarios), for which a joint convergence analysis was not available. These theoretical results also demonstrate the  potential of the \gls{END} notation for analysis purposes, even in problems without any sparsity (Section \ref{subsec:NEalgorithms});
 \item For \gls{GNE} problems, we study a novel class of aggregative games, generalizing that considered in \cite{Eslami:Unicast:TAC:2022}. We demonstrate that the amount of copies allocated in the network, for both the  aggregation function and the dual variables, can be reduced according to the problem sparsity. In particular, for the first time, we explicitly account for the possible sparsity in the coupling constraints (Section~\ref{subsec:GNEalgorithms}). Finally, we test our algorithm against its sparsity-unaware counterpart in a unicast rate allocation application.
%
%
Our simulations show that not only  communication and memory overhead are significantly reduced, but also that convergence 
speed and scalability are
improved (Section~\ref{sec:numerics}).
\end{enumerate}
%

\smallskip
We focus on games under partial-decision information here; we refer the  interested reader to the extended draft of this paper \cite{Bianchi_minG_TAC_2022} for numerical and theoretical results of \gls{END} algorithms for multi-agent optimization.
\vspace{-0.2em}

\subsection{Background} \label{sec:background}
\vspace{-0.2em}
\subsubsection{Basic notation}  $\mathbb{N}$ is the set of natural numbers, including $0$.
$\R$ ($\R_{\geq 0}$) is the set of (nonnegative) real numbers.
$\0_q\in \R^q$  ($\1_q\in\R^q$) is a vector with all elements equal to $0$ ($1$); $\0_{q\times p} \in \R^{q\times p}$ is a matrix with all elements equal to $0$;  $\mathrm{I}_q\in\R^{q\times q}$ is an identity matrix; the subscripts may be omitted when there is no ambiguity. $e_i$ denotes a vector of appropriate dimension with $i$-th element equal to 1 and all other elements equal to 0.
For  a matrix $ \mat A  \in \R^{p \times q}$, $[\mat A]_{i,j}$ is the element on  row $i$ and column $j$; $\Null(\mat A)\coloneqq \{x\in\R^q \mid \mat Ax=\0_n\}$ and $\range(\mat A)\coloneqq \{v\in\R^p \mid v=\mat Ax, x\in\R^q \}$; $\| \mat A \|_\infty$ is the maximum of the absolute row sums of $\mat A$.
If $ \mat A= \mat A^\top\in\R^{q\times q}$, $\uplambda_{\textnormal{min}}(\mat A)=:\uplambda_1(\mat A)\leq\dots\leq\uplambda_q(\mat A)=:\uplambda_{\textnormal{max}}(\mat A)$ denote its eigenvalues.
$\diag(\mat A_1,\dots,\mat A_N)$ is the block diagonal matrix with $\mat A_1,\dots,\mat A_N$ on its diagonal. Given $N$ vectors $ x_1, \ldots, x_N$,  $\col (x_1,\ldots,x_N ) \coloneqq  [ x_1^\top \ldots  x_N^\top ]^\top$. $\otimes$ is the Kronecker product.

\subsubsection{Graph theory} \label{sec:graphtheory}
A (directed) graph $\g{}=(\mcV,\E{})$
consists of a nonempty set of vertices (or nodes) $\mcV=\{1,2,\dots, V \}$ and a set of edges $\E{}\subseteq \mcV \times \mcV $. We denote by $\neig{}{v}\coloneqq \{ u\mid (u,v) \in \E{} \} $ and $\neigo{}{v}\coloneqq \{ u \mid (v,u) \in \E{} \} $ the set of in-neighbors (or simply neighbors)  and out-neighbors of vertex $ v \in \mcV$, respectively.
A path  from $v_1\in \mcV $ to  $v_N\in \mcV$ of length $T$ is a sequence of vertices $(v_1,v_2,\dots,v_T)$ such that $(v_t,v_{t+1} )\in \E{}$ for all $t=1,\dots, T-1$. $\g{}$ is rooted at $v\in \mcV{}$ if there exists a path  from $v$ to each   $u \in \mcV{} \backslash \{v\}$. $\g{}$ is strongly connected if there exist a path from $u$ to $v$, for all $u,v \in \mcV{}$; in case $\g{}$ is undirected, namely if $(u,v) \in\E{}$ whenever $(v,u) \in\E{}$, we simply say that $\g{}$ is connected. The restriction of the  graph $\g{}$ to a set of vertices $\V{A} \subseteq \V{}$ is defined as $\g{}|_{\V{A}} \coloneqq ( \V{A}, \E{} \cap (\V{A} \times \V{A}))$.   We also write $\g{}=(\V{A}, \V{B},\E{})$ to highlight that $\mc{G}$ is bipartite, namely $\mc{G}=(\mcV,\E{})$ with  $\V{} = \V{A} \cup \V{B}$ and $\E{} \subseteq \V{A} \times \V{B}$.
We may associate to  $\g{}$ a weight matrix $\W{} \in \R^{V\times V}$ compliant with $\g{}$, namely $\w{}_{u,v}\coloneqq[\W{}]_{u,v}>0$ if $(v,u)\in \E{}$, $\w{}_{u,v}=0$ otherwise;
we denote by
$\D{}=\diag((\textrm{deg}{}(v))_{v\in\mcV{}})$ and $\L{}=\D{}-\W{}$ the in-degree  and Laplacian matrices, with $\textrm{deg}(v)=\textstyle \sum _{u\in\mcV{}} \w{}_{u,v}$.  $\g{}$ is unweighted if $\w{}_{u,v}=1$ whenever $(v,u) \in \E{}$. 
Given two graphs $\g{A}=(\V{A},\E{A}) $ and $\g{B}=(\V{B},\E{B}) $, we write  $\g{A} \subseteq \g{B}$ if $\g{A}$ is a subgraph of $\g{B}$, \ie if $\V{A} \subseteq \V{B}$ and $\E{A} \subseteq \E{B}$; we let $\g{A} \bigcup \g{B} \coloneqq  (\V{A} \cup \V{B}, \E{A} \cup \E{B})$. Given  a graph $\g{}=(\mcV,\E{})$, compliant weights $\W{}$, terminals $\mc T\subseteq \mcV{}$,
 we define the following problems:
\begin{itemize}[leftmargin=*]
	\item   Steiner tree ST$(\g{},\mc T,\W{})$: find an undirected connected subgraph $\g{*}=(\V{*},\E{*})\subseteq \g{}$, $\mc T\subseteq \V{*}$, with minimum cost (\ie minimizing $\sum_{(v,u)\in \E{*}} \w{}_{u,v}$);
	\item Unweighted Steiner tree  UST$(\g{},\mc T)$: find an undirected connected subgraph $\g{*}=(\V{*},\E{*})\subseteq \g{}$, $\mc T\subseteq \V{*}$, with minimum number of nodes (equivalently, of edges).
\end{itemize}
\subsubsection{Euclidean spaces} Given a positive definite matrix $ \R^{q\times q} \ni \mat Q \succ 0$,  $\mc{H}_{\mat Q}\coloneqq (\R^q,\langle \cdot \mid \cdot \rangle _\mat{P})$ is the Euclidean space obtained by endowing $\R^q$ with the
$\mat{Q}$-weighted inner product  $\langle x \mid y \rangle _\mat{Q}=x^\top \mat Q y$,  $\| \cdot \|_{\mat Q}$ is the associated norm; we omit the subscripts if $\mat Q = \id $. Unless otherwise stated, we  assume to work in $\Hmc = \mc{H}_{\mat{I}}$.

\subsubsection{Operator theory}
A set-valued  operator $\op {F}:\R^q\rightrightarrows \R^q$ is characterized by its graph
$\gra (\op {F})\coloneqq \{(x,u) \mid u\in \op{F}(x)\}$. $\dom(\op{F})\coloneqq \{x\in\R^q \mid  \op{F}(x)\neq \varnothing \}$,
$\fix\left( \op{F}\right) \coloneqq  \left\{ x \in \R^q \mid x \in \op{F}(x) \right\}$ and $\zer\left( \op{F}\right) \coloneqq  \left\{ x \in \R^q \mid \0 \in \op{F}(x) \right\}$ are the domain, set of fixed points and set of zeros, respectively. $\op{F}^{-1} $ denotes the inverse operator of $\op{F}$, defined as $\gra (\op{F}^{-1})=\{(u,x)\mid (x,u)\in \gra(\op{F})\}$.
$\op{F}$ is
($\mu$-strongly) monotone if $\langle u-v \mid x-y\rangle \geq 0$ ($\geq \mu\|x-y\|^{2}$) for all $(x,u)$,$(y,v)\in\gra(\op {F})$.
$\Id$ is the identity operator. 
For  a set  $\Omega \subseteq \R^q$, 
$\nc_{\Omega}: \Omega \rightrightarrows \R^q:x\mapsto \{ v \in \R^q \mid \sup_{z \in \Omega} \, \langle v \mid z-x \rangle \leq 0  \}$ is the normal cone operator of $\Omega$. If $\Omega$ is closed and convex, then 
$(\Id+\nc_\Omega)^{-1}=\proj_\Omega$ is the Euclidean projection onto  $\Omega$.

\section{Estimate network \design{}  }\label{sec:mathbackground}

In this section, we introduce
the general \gls{END} framework and notation. 
We then leverage these tools to develop novel and improved \gls{GNE} seeking algorithms in \cref{sec:GNE}.

\subsection{The \gls{END} setup}\label{subsec:setup}

We start by defining a generic
information structure, 
useful  to describe existing distributed algorithms and to design new ones. It is characterized by:

\begin{itemize}
	\item  a set of agents $\I \coloneqq\{ 1,2,\dots,N\}$;
	\item a given (directed) \emph{communication} network $\g{C}=(\I,\E{C})$, over which the agents can exchange information:  agent $i$ can receive information from agent $j$ if and only if
	$j\in\neig{C}{i}$;
	\item a variable of interest $y \in \R^ \n{y}$, partitioned as $y=\col((y_p)_{p\in \P})$, where
    $y_p\in\R^{\n{y_p}}$
    and $\P\coloneqq\{1,2,\dots,P\}$;
	\item a given bipartite directed \emph{interference} graph $\g{I}=(\P,\I,\E{I})$,  $\E{I}\subseteq \P \times \I$, that specifies which components of $y$  are indispensable for each agent:  $p\in \neig{I}{i}$ means that agent $i$ needs (an estimate of) $y_p$ to perform some essential local computation.\footnote{For ease of notation, we  assume that $\neigo{I}{p}\neq \varnothing$ for all $\p$ (i.e., each component of $y$ is indispensable for some agent).}
\end{itemize}
The agents may be unable to  access the value of the variable of interest $y$. Instead, each agent keeps an estimate of some (possibly all) the components $y_p$'s, and exchanges its estimates with some neighbors, as specified by:

\begin{itemize}
	\item a bipartite directed \emph{estimate} graph $\g{E}=(\P,\I,\E{E})$, $\E{E}\subseteq \P \times \I$, that specifies which components of $y$ are estimated by each agent: agent $i$ keeps an estimate $\h{y}_{i,p}\in\R^{\n{y_p}}$ of $y_p$ if and only if $p\in \neig{E}{i}$, and thus $\neigo{E}{p}$ is the set of agents that keeps an estimate of $y_p$;
	\item $P$ directed \emph{\design{}} graphs  $\{\gD{p} \}_{\p}$, with  $\gD{p}=(\neigo{E}{p},\ED{p})$, to describe how the agents exchange their estimates: agent $i$ can receive $\h{y}_{j,p}$ from agent $j$ if and only if $j\in\neigD{p}{i}$.
\end{itemize}

Note that the vertices of $\gD{p}$ are $\neigo{E}{p}$, namely only the agents that keep an
estimate of $y_p$ could receive an estimate of $y_p$. Furthermore, note that the graph $\g{E}$ is uniquely determined by  $\{\gD{p} \}_{\p}$ (but not vice versa).

%

\begin{example}[Partially separable costs] \label{ex:do1}  
Consider a multi-agent problem, where the agents can only exchange information over a   
 communication network $\g{C}$, and each agent $\i$ has a private cost function $f_i(x)$, depending on a global decision vector $x\in \R^\n{x}$; for instance an optimization problem (where the goal is to minimize $\sum_{\i} f_i$) or a game as in \eqref{eq:gameintro}. Set $y= x$ and partition $y = (y_p)_{p\in\mc{P}}$ in $P$ components. 
In several engineering applications, like network control and data ranking \cite{NecoaraClipici_Coordinate_SIAM2016},
each cost function $f_i$  depends only on some of the components of $y$, as it can be specified by an interference graph $\g{I}$: $f_i$ depends on $y_p$ if and only if $p\in\neig{I}{i} \subseteq\P$. With some abuse of notation, we emphasize
 this fact by writing 
	\begin{align}
		f_i(y)=f_i ((y_p)_{p\in \neig{I}{i}}) . \label{eq:fipartiallyseparable}
	\end{align}
In this setup, the usual approach to solve distributed optimization \cite{Nedic_DIGing_SIAM2017} and \gls{NE}  problems \cite{YeHu_Consensus_TAC2017}  is to  assign to each agent $\i$ an estimate $\tilde{\hy}_i\coloneqq \col ((\h{y}_{i,p})_\p) \in \R^\n{y} $ of the whole decision variable 
	and to let the agents exchange their estimates with every neighbor over $\g{C}$. In \gls{END} notation, we write this as\footnote{In the following, we  often refer to \eqref{eq:standard} as the  ``standard'' setup, since it is the usually studied scenario. }
	\begin{align}	\label{eq:standard} \tag{C0}
		\E{E} =\P \times \I, \qquad  
		\gD{p} =\g{C} \  (\forall \p).
		\end{align}
This choice of graphs $\g{E}$ and $\{\gD{p}\}_{\p}$ ignores the possible structure in \eqref{eq:fipartiallyseparable}. Note that
  agent $i$ only needs  $(y_p)_{p\in \neig{I}{i}}$ to evaluate (the gradient of) its local cost $f_i$. Therefore, storing an estimate of the whole vector $y$ could be unnecessary and  inefficient -- especially if  $\g{I}$ is sparse and  $P$ is large. Are smarter choices of  $\g{E}$ and $\{\gD{p}\}_{\p}$  possible?
  
  On a side note, we mention that different partitions for the variable $y$ are possible; for example, in the case of the game in \eqref{eq:gameintro}, one natural choice is to set  $\mc{P}= \mc{I}$ and $y_p = x_p$ (the local action of agent $p$), for which \eqref{eq:fipartiallyseparable} retrieves exactly \eqref{eq:partiallycoupledcost}.
	\hfill $\square$
\end{example}
\color{black}

\subsection{Flexible design}\label{subsec:design}

From an algorithm deployment perspective, the  the graphs $\g{C}$ and $\g{I}$ shall be considered  fixed a priori. In contrast,  the graphs $\g{E}$ and $\{\gD{p} \}_{\p}$ are a design choice, which determines how the estimates of certain variables are actually exchanged during the execution of a distributed algorithm. 
 Informally speaking, one wishes to design the graphs $\g{E}$ and $\{\gD{p} \}_{\p}$ so that it is possible to \emph{distributedly} and \emph{efficiently} solve a given decision problem. Mathematically, this translates to imposing  extra structure on the estimate and \design{} graphs.

\smallskip
\begin{pb}\label{prob}
	Given the communication graph $\g{C}$ and the interference graph $\g{I}$, design the estimate graph $\g{E}$ and the design graphs $\{\gD{p} \}_{\p}$ such that:\nopagebreak
	\begin{thmlist}
		\item \label{prob:1} $\g{I}\subseteq\g{E}$;
		\item  \label{prob:2} $\gD{p}\subseteq \g{C}$, for all $\p$;
		\item  \label{prob:3} ``additional requirements'' on $\g{E}$, $\{\gD{p} \}_{\p}$ are met. ~\hfill  $\square$
	\end{thmlist}
\end{pb}
\smallskip

\setlength{\belowcaptionskip}{-3pt}

\begin{figure}[t]
\begin{subfigure}[a]{\columnwidth}
\centering
\includegraphics[width=0.9\columnwidth]{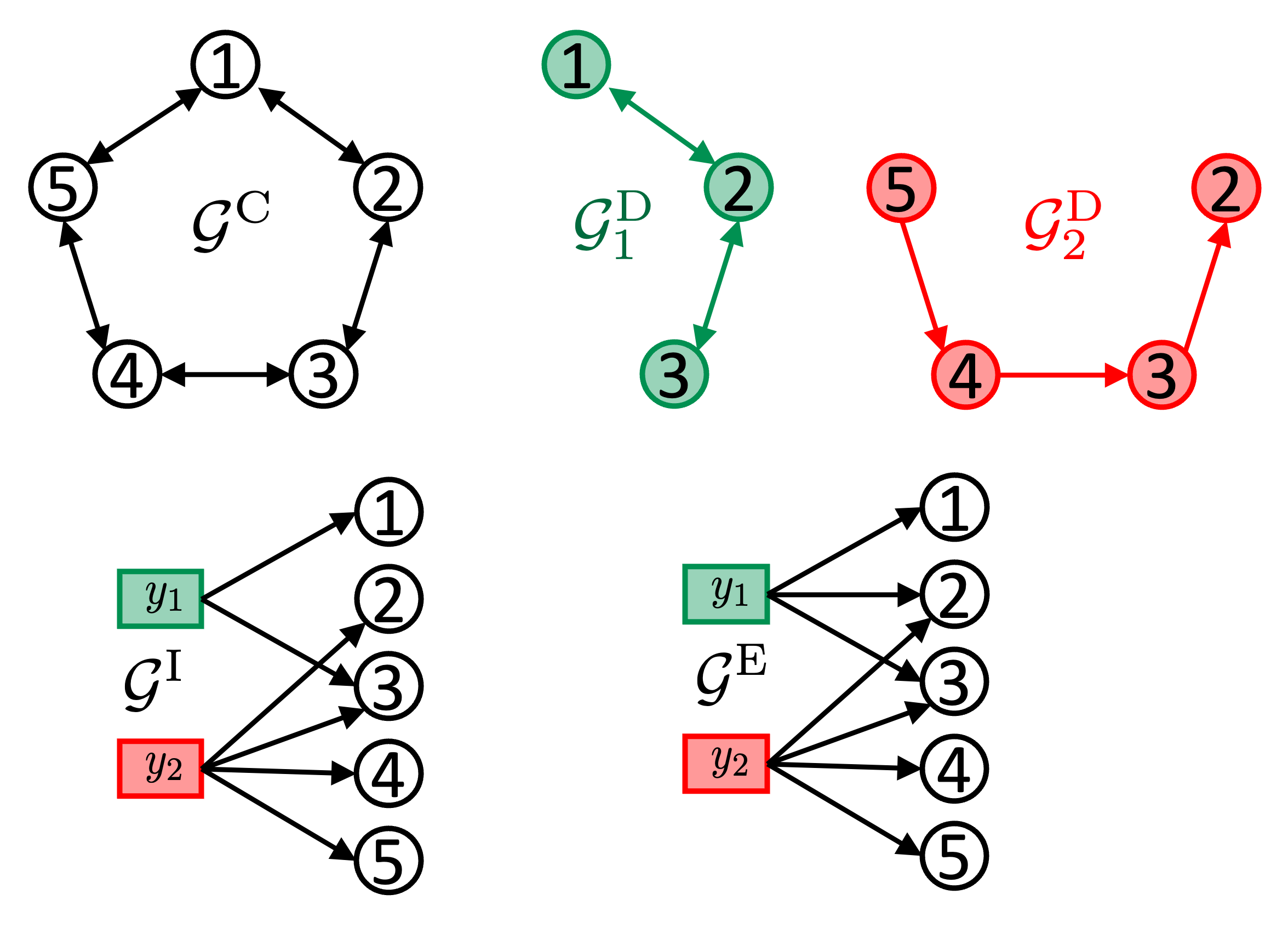}
\caption{}
\label{fig:0:A}\end{subfigure}
\begin{subfigure}[b]{\columnwidth}
\centering
\includegraphics[width=0.5\columnwidth]{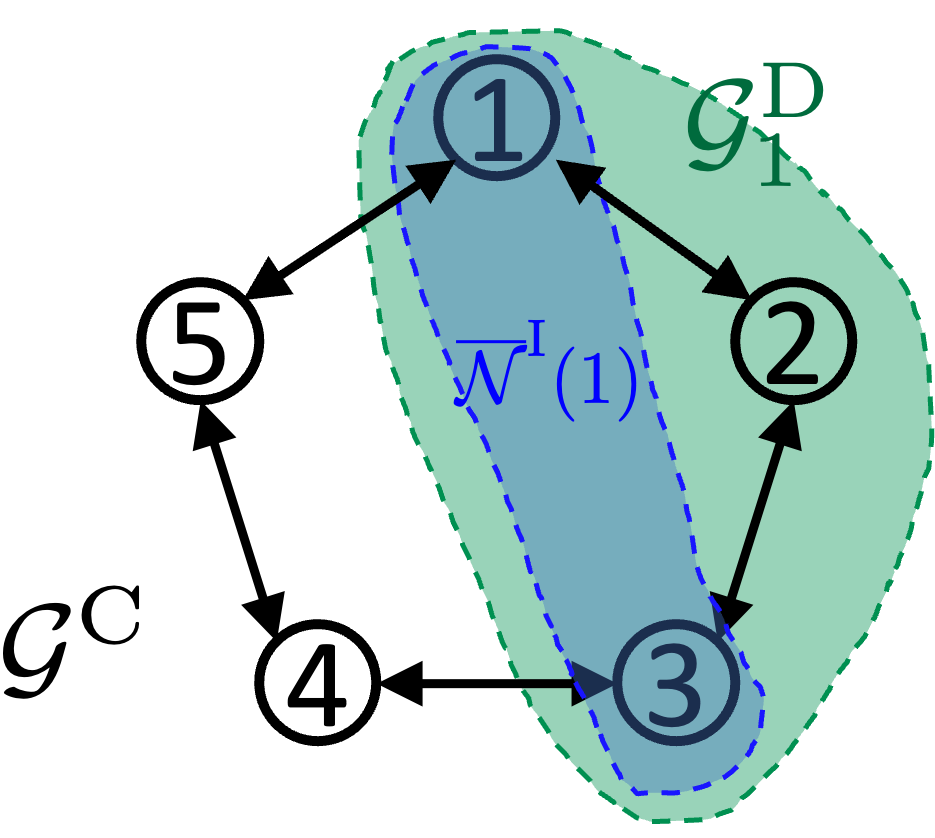}
\caption{}
\label{fig:0:B}
\end{subfigure}
\caption{\label{fig:0}
 (a) A simple example of \gls{END} design. On the left, the given communication and interference graphs, with $\mc{I} = \{1,2,3,4,5\}$ and $\mc{P} = \{1,2\}$. On the right a possible choice for the design graphs and the corresponding estimate graphs; they solve \cref{prob} with ``\emph{(iii)} $\gD{1}$ is undirected and connected and with minimal number of nodes; $\gD{2}$ is rooted at $5$ and with minimal number of edges''.  
 \newline 
 (b) Let us focus on the design of $\gD{1}$. The goal is to minimize the number of copies of $y_1$ (i.e., the number of nodes in $\gD{1}$), but provided that $\gD{1}$ is connected and the conditions in \cref{prob:2,prob:1} are met (i.e., $\neigo{I}{1} \subseteq \neigo{E}{1}$ and $\gD{1} \subseteq \g{C}$). Note that agent $3$ has to estimate  $y_1$ (\ie $3 \in \neigo{E}{1}$), even though agent $3$ is not directly affected by $y_1$ (\ie $3 \notin \neigo{I}{1}$): otherwise, the information could not travel between nodes $1$ and $2$, which are not communication neighbors. In general, a solution to this design problem can be obtained by solving the Unweighted Steiner Tree problem UST$(\mc{G}^C,\neigo{I}{1})$ (see \cref{sec:background}), for which distributed off-the-shelf algorithms are available \cite{Chalermsook:Distributed:Steiner:2005}. The problem is NP-hard in general, but can be approximated in polynomial time. Sufficient for the existence of a solution is that $\g{C}$ is undirected and connected.
}
\end{figure}

In particular, it must hold that $\g{E}\subseteq{\g{I}}$, namely each agent estimates at least the components of $y$ which are indispensable for local computation (in fact, $\g{I}$ expresses the minimal information necessary for each agent). Moreover, since data transmission can only happen over the communication graph $\g{C}$, it must hold that $\gD{p}\subseteq \g{C}$, for all $\p$.

The ``additional requirements'' in \cref{prob:3} can encode \emph{feasibility} conditions (seen as hard constraints: for instance, some type of connectedness  for the graphs $\{\gD{p} \}_{\p}$ is always needed, to ensure that the agents can reach consensus on their estimates), but also  \emph{efficiency} specifications (treated as soft constraints, e.g, one might aim at reducing  memory allocation by minimizing the  number of copies of each variable used, provided that all  hard constraints are satisfied).

A very simple design example  is presented in \cref{fig:0}. For continuity of presentation, we discuss in details in \cref{app:tutorial} several  instances of \cref{prob}. 

\begin{remark}[On the design cost]
An optimal design of the graphs $\g{E}$ and $\gD{p}$'s can be
computationally demanding.
We emphasize that such design is still part of the tuning of an algorithm, to be subsequently used to solve an underlying decision problem. The question is thus whether the
benefits of an efficient estimate allocation -- in terms of algorithm execution  -- are worth  the additional initial design effort.

For time-varying \cite{Bianchi_tvGNE_CDC2023} and \emph{repeated} problems, like distributed optimal estimation or model predictive control  \cite{Mota:LocalDomains:TAC:2014},  (where the same distributed problem is solved multiple times, but for different values of some parameters/measurements, see also \cref{sec:numerics}),
a careful a priori design can be advantageous, i.e., worth the initial (one-time) computational cost of solving \cref{prob}. Otherwise, especially if $\g{I}$ is very dense, it may be convenient  to settle for a suboptimal,
but readily available choice -- e.g.  \eqref{eq:standard}, which is the standard solution in literature \cite{Scutari_Unified_TSP2021,Pavel_GNE_TAC2020}.
Finally,  in many relevant applications,
the specific problem structure renders the choice of optimal  graphs $\gD{p}$  straightforward, as shown in \cref{subsubsec:straightforward}. \hfill $\square$
\end{remark}
\color{black}



\subsection{Unified analysis}

We emphasize that solving \cref{prob} is not the  goal of this paper. In fact,  the requirements and design procedure for $\g{E}$ and $\{\gD{p}\}_\p$ are vastly problem-dependent. Luckily we do not need to consider a specific structure for the estimate and design graphs, but just assume that they satisfy some properties. 
First, 
we will assume throughout the paper  that the design graphs are chosen to satisfy the specifications in \cref{prob:1,prob:2}, without further mention.
 \begin{standing}[Consistency]\label{asm:consistency}
	It holds that $\g{I}\subseteq \g{E}$ and that $\gD{p}\subseteq \g{C}$  for all $\p$.  \hfill $\square$
\end{standing}

Second, we will simply assume some level of connectedness for the design graphs; for example, we might assume that each $\gD{p}$ is strongly connected. 
This ensures some properties for the estimate exchange, akin to those exploited in the analysis of standard consensus-based algorithms, as exemplified in \cref{lem:rooted,lem:connected} below. In turn, this simple observation allows one to easily generalize the convergence analysis of  several  distributed algorithms to the \gls{END} framework, allowing for greater  freedom in the estimates exchange design.

However, before going deeper into details, we need to introduce the stacked notation used in the remainder. This notation is particularly important in simplifying our analysis, as it allows us to  abstract from the complex network interaction and to seamlessly cope with  the non-homogeneity of the agents’ copies (e.g. the local vectors kept by distinct agents may have different dimensions/number of components).
\color{black}



\subsection{\gls{END} Notation}\label{subsec:notation}
For all $\p$, let $\Ntilde{p}\coloneqq\left| \neigo{E}{p} \right|$ be the overall number of copies of $y_p$ kept by all the agents. We define:
\begin{align}
	\hy_p & \coloneqq \col ( ( \h{y}_{i,p} )_{i\in\neigo{E}{p}} )\in \R^{  \Ntilde{p} n_{y_p}}, \quad \forall \p;
	\\
	\hy   & \coloneqq\col( ( {\h{y}}_p )_{\p} ) \in \R^\n{\hy},
\end{align}
where we recall that $\hy_{i,p}$ is the estimate of the quantity $\y_p$ kept by agent $i$; $\n{\hy}\coloneqq\sum_{\p} \Ntilde{p} \n{y_p}$.  Note that $\hy_p$ collects all the copies of $y_p$, kept by different agents. 

We  denote 
\begin{align}
    \WDbs &\coloneqq \diag ( ( \WD{p}\otimes \textrm{I}_{\n{y{_p}}})_{\p} ) 
    \\
    \LDbs &\coloneqq \diag ( ( \LD{p}\otimes \textrm{I}_{\n{y{_p}}})_{\p} ),
\end{align} where $ \WD{p}$ is the weight matrix of $\gD{p}$, and $\LD{p}$ is its Laplacian. For all $\p$, let
\begin{align}
	\Ebs{p} &\coloneqq\{\hy_p \in \R^{N_p \n{y_p}} \mid \hy_p = \1_{N_p} \otimes v, v\in\R^{\n{y_p}} \},
 \\
 \Ebs{} &\coloneqq \textstyle \prod_{\p} \Ebs{p}
\end{align}
be the consensus space for $\hy_p $, namely the subspace where all the estimates of $\y_p$ are equal, and 
the overall consensus space, respectively. Given $y=\col(( \y_p)_{\p} )\in \R^\n{y}$, we also define 
\begin{align}\label{eq:Cconsensus}
    \Ebs{}(\y)\coloneqq\col(( \1_{N_p} \otimes \y_p )_{\p} ) \in \Ebs{}.
\end{align}
We denote by $\Ebsperp{}$ the complementary subspace of $\Ebs{}$; by $\Piparallel{} \coloneqq \diag ( (
\1_{Np} \1_{N_p}^\top \otimes I_{\n{y_p}} / N_p )_{\p} )$ and $\Piperp{}\coloneqq\id-\Piparallel{}$  the projection matrices onto $\Ebs{} $ and $\Ebsperp{}$, respectively.

For each $\p$, for each $i\in \neigo{E}{p}$, we denote by \begin{align}\label{eq:ip}
    i_p\coloneqq \textstyle \sum_{j\in \neigo{E}{p}, \, j\leq i} 1
\end{align} the  position of $i$ in the ordered set of nodes  $\neigo{E}{p}$. For all $i \in \neigo{E}{p}$, we denote by $\Rmc_{i,p}\in \R^{\n{y_p} \times {N_p \n{y_p}}}$ the matrix that selects $\hy_{i,p}$ from $\hy_p$, \ie  $\hy_{i,p}= \Rmc_{i,p} \hy_p$.

Sometimes it is useful to define agent-wise quantities, which we indicate with a tilde. Let
\begin{align}
	\hyt_i &\coloneqq  \col ( ( \hy_{i,p} )_{p\in \neig{E}{i}} ), \quad \forall \i;
	\\
	\hyt &\coloneqq\col( ( \hyt_i )_{\i} ) \in\R^{\n{\hy}},
\end{align}
where $\hyt_i$ collects all the estimates kept by agent $i$. Let $\perm \in \R^{\n{\hy}\times \n{\hy}}$ be the permutation matrix such that $\perm \hy= \hyt$; the graph structures corresponding to $\hyt{}$ can be defined via permutations, e.g.,  $\WDtbs\coloneqq\perm \WDbs \perm^\top$,  $\LDtbs\coloneqq\perm \LDbs \perm^\top$, $\Etbs{} \coloneqq\perm \Ebs{}$.

The following lemmas easily follow by stacking over $\p$ well-known graph-theoretic properties.

\begin{lemma}\label{lem:rooted}
	Assume that, for each $\p$, there exists a root $r_p \in \mc{I}$ such that $\gD{p}$ is rooted at $r_p$ . Then, $\Null(\LDbs) =\Ebs{} $. Moreover, $\Null(\LDtbs)=\Etbs{} $.  \hfill $\square$
\end{lemma}


\begin{lemma}\label{lem:connected}
	Assume that $\gD{p}$ is strongly connected and that $\WD{p}$ is balanced, for all $\p$. Then, for any $\hy \in\R^{\n{\hy}}$, 
    $
	\langle \hy, \LDbs \hy \rangle \geq \frac{\bar{\uplambda}}{2} \| (\id -\Piparallel{} ) \ybs \|^2,
	$
	where $\bar{\uplambda}\coloneqq\min_{\p}\{ \uplambda _2 ({\LD{p}}^\top+\LD{p}) \} >0$. \hfill $\blacksquare$
\end{lemma}


\section{Generalized Nash equilibrium seeking}\label{sec:GNE}
In this section we consider  \gls{GNE} problems. In particular, each agent $\i$ is equipped with a private cost function $f_i(x_i,x_{-i})$, $f_i:\R^{\n{x_i}}\times \R^\n{x_{-i}} \rightarrow  {\R}$,
$
$
which depends  both on its local  action (decision variable) $x_i \in \R^{\n{x_i}}$ and on the actions of the other agents $x_{-i}\coloneqq \col((x_j)_{j\in\I \backslash \{ i \} } ) \in \R^{\n{x_{-i}}}$. Each agent chooses its action in a local feasible set $\Omega_i \subseteq \R^\n{x_i}$;
let  $x \coloneqq \col( (x_i)_{\i})  \in \Omega $ be the overall action, with  $ \Omega \coloneqq \prod_{\i} \Omega_i \subseteq \R^\n{x}$. The agents' decisions are also coupled via shared constraints: specifically, the overall feasible set is $\X{}  \coloneqq \Omega \cap \{ x\in \R^{\n{x}}\mid \mat Ax \leq a\}$, where $\mat A \in \R^{\n{\lambda}\times \n{x}}$, $a\in \R^{\n{\lambda}}$. We call \emph{generalized game} the following set of interdependent optimization problems:
\begin{align}\label{eq:game}
	(\forall \i ) \ \minimize{x_i \in \R^{\n{x_i}}} \ f_i(x_i,x_{-i}) \ \text{s.t. } (x_i,x_{-i}) \in \X.
\end{align}
The   goal is to distributedly compute a \gls{GNE}, a set of decisions   simultaneously solving all the problems in \eqref{eq:game}. \begin{definition}
    A \gls{GNE} is an  $N$-tuple $x^\star =\col( (x_i^\star)_{\i}) \in \X{} $ such that, for all $\i$,   $f_i(x_i^\star, x_{-i}^\star) \leq \inf_{x_i} \{ f_i(x_i , 	x_{-i}^\star) \mid  (x_i, x_{-i}^\star) \in \X \}$. \hfill $\square$
\end{definition}  Let us define the \emph{pseudo-gradient} operator $\op F:\R^{\n{x}} \rightarrow \R^{\n{x}}$,
	\begin{align}\label{eq:pseudo-gradient}
		\op{F}(x)\coloneqq \col(( \grad{x_i} f_i(x_i,x_{-i}) )_{\i}).
	\end{align}
We restrict our attention to convex  
 and strongly monotone games; the following are standard conditions for \gls{GNE} seeking over graphs \cite[Asm.~1]{KoshalNedicShanbhag_Aggregative_OR2016}, \cite[Asm.~1-2]{TatarenkoShiNedic_Geometric_TAC2021}, \cite[Asm.~1-2]{Pavel_GNE_TAC2020}.
\begin{assumption}
\label{asm:game_convexity}
	For all $i\in \mathcal{I}$, $\Omega_i$ is closed and convex,
    $f_{i} $ is continuous and
	$f_{i}(\cdot, x_{-i})$ is convex
    and differentiable	for every $x_{-i}$; $\mc{X}$ is non-empty
	and satisfies Slater's constraint qualification.
	{\hfill $\square$}
\end{assumption}
\begin{assumption}\label{asm:strong_mon}
	The pseudo-gradient $\op F$ in \eqref{eq:pseudo-gradient} is $\mu$-strongly monotone and $\theta$-Lipschitz continuous, for some $\mu,\theta>0$.  \hfill $\square$
\end{assumption}

As per standard practice, we only focus on \glspl{v-GNE}, namely \glspl{GNE} with identical dual variables, which are computationally tractable and  economically more justifiable \cite{FacchineiKanzow_Generalized_4OR2007}. 
Under \cref{asm:game_convexity,asm:strong_mon}, there is a unique \gls{v-GNE} \cite[Th.~2.3.3]{FacchineiPang_VIs_Springer2004}; moreover $x^\star$ is the \gls{v-GNE} of the game in \eqref{eq:game} if and only if  there exists a  dual variables $\lambda^\star \in \R^{n_\lambda}$  satisfying the  following \gls{KKT} conditions \cite[Th.~4.8] {FacchineiKanzow_Generalized_4OR2007} (we recall that $\nc_{\mc{S}}$ denotes the normal cone of a set $\mc{S}$):
\begin{subequations}
\label{eq:KKT}
\begin{align} \label{eq:KKTa}
	\0_{n_{x}} & \in  \op F(x^\star) +  \nc_{\Omega}(x^\star) +\mat A^\top \lambda^\star , 
 \\
 \label{eq:KKTb}
  \0_{n_{\lambda}}  & \in  \mat -(Ax^\star -a) + \nc_{\R^\n{\lambda}_{\geq 0}}.
\end{align}
\end{subequations}


Finally, we  consider the so-called partial-decision information scenario,
arising in applications without a central coordinator, such as radio networks \cite{Wang_CognitiveRadio_2010} and sensor positioning \cite{Romano_Penalty_TCNS2023}. In particular, each agent  $i$
only relies  on the data received locally from some neighbors over a communication network $\g{C}=(\I,\E{C})$. To cope with the limited information, the solution usually explored in the literature  is to embed each agent with an estimate of the whole vector $x$  \cite{YeHu_Consensus_TAC2017}, \cite{TatarenkoShiNedic_Geometric_TAC2021}, and possibly a copy of a dual variable  \cite{Pavel_GNE_TAC2020}. Critically, this approach fails to exploit the possible sparsity in the cost and constraint coupling. We remedy in the remainder of this section. 

\subsection{\gls{END} pseudo-gradient dynamics for NE seeking }\label{subsec:NEalgorithms}

We first consider games without coupling constraints (i.e., $\mc{X}=\Omega)$: then, the notion of \gls{v-GNE} boils down to that of a \gls{NE}. 
We describe the cost coupling  via an interference graph $\g{I} =(\I,\I, \E{I}) $,   where  $(p,i) \in \E{I}$ if and only if  $f_i$ depends on $x_p$, for all $i \neq p$, and $(i,i)\in  \E{I}$ for all $\i$; we also write
\begin{align*}
	f_i((x_p)_{p\in\neig{I}{i}})\coloneqq f_i(x_i,x_{-i}).
\end{align*}

Hence, in this subsection we choose  the variable of interest for the \gls{END} framework to be the overall action, i.e. $y=x$; $\P=\I$ and $y_i=x_i$ for all $\i$ (finer partitions are  also possible). Assume that an estimate graph $\g{E}$ and design graphs $\g{D}$'s are chosen according to \cref{asm:consistency}. Then, each agent $i$ keeps and sends the copies $\{ \hy_{i,p}, p\in \neig{E}{i} \}$, estimating the actions of a \emph{subset} of the other agents.  Since the action $x_i$ is actually a local variable, under the control of agent $i$, we formally define $\hy_{i,i}\coloneqq x_i$ (i.e., agent $i$'s estimate of its own action coincides with the \emph{real} value). We study the following iteration (we recall the notation in \cref{subsec:notation}): each agent $\i$ performs
	\begin{align}\label{eq:proxgrad}
	    \nonumber
		\hat\hy_{i,p}^{k }  \coloneqq{} &
		\textstyle \sum_{j \in \neigD{p}{i}}	[\WD{p}]_{i_p,j_p} \hy_{i,p}^{k}   
		 \quad  \
		(\forall p \in \neig{E}{i})
		\\
		\nonumber
		\hy_{i,p}^{k+1} ={} & \hat \hy_{i,p}^{k} 
		\hphantom{\textstyle \sum_{j \in \neigo{D}{p }{i}}	[\WD{p}]_{i_p,j_p} }
		 \quad  \  (\forall p \in \neig{E}{i} \backslash \{i\})
		\\
		\hy_{i,i}^{k+1}  ={} & \proj_{\Omega_i}  \left ( \hat\hy_{i,i}^{k} - \alpha \grad{x_i} f_i ( (\hat\hy_{i,p}^k)_{p\in  \neig{I}{i} }) \right).
	\end{align}
	In \eqref{eq:proxgrad}, the estimates of the agents are updated according to a consensus protocol, with an extra (projected) gradient step for the own estimate $\hy_{i,i}$. The algorithm boils down to   \cite[Alg.~1]{Bianchi_Timevarying_LCSS2021} if $\WD{i}= \W{C}$ for all $\i$, (\ie the standard setup in \eqref{eq:standard}). Since $\gD{i}\subseteq \g{C}$, the algorithm is distributed.
Note  that the local gradient $\grad{x_i}f_i$ is computed on the local estimates kept by agent $i$, not on the real action $x^k = \col(( \hy_{i,i}^k )_{\i} )$.  We define the \emph{extended pseudo-gradient mapping}
\begin{align}\label{eq:extended-pseudo-gradient}
	\Fbs (\hy ) \coloneqq \col(( \grad{x_i} f_i( (\hy_{i,p})_{p\in  \neig{I}{i} })  )_{\i}),
\end{align}
$\Rmc\coloneqq\diag((\Rmc_{i,i} )_{\i})$, with $\Rmc_{i,i}$ as in \cref{subsec:notation} (i.e., $\Rmc\hy = y $), and  $\bs{\Omega} \coloneqq \{ \hy \mid \Rmc \hy \in \Omega \}$.
Then, \eqref{eq:proxgrad}  can be written in stacked form in one line:
\begin{align}\label{eq:proxgrad_compact}
	\hy^{k+1} = \proj_{\bs{\Omega}} \left (  \WDbs \hy^k - \alpha  \Rmc ^\top \Fbs (   \WDbs \hy^k) \right).
\end{align}
We now postulate additional conditions on the design graphs.
\begin{assumption}\label{asm:row_spanning}
	For each $\i$, $\gD{i}$ is rooted at $i$ and $\WD{i}\1_{N_i}=\1_{N_i}$; we denote by $\qD{i}\in \R^{N_i}$ the unique nonnegative vector such that ${\qD{i}}^\top \WD{i} ={\qD{i} } ^\top$, $\1^\top_{N_i} \qD{i}=1$.  \hfill $\square$
\end{assumption}
 \cref{asm:row_spanning} is very mild:
rootedness is necessary for the consensus of the estimates;
row-stochasticity  is immediately satisfiable whenever the agents have access to their own in-degree. One major technical complication -- with  respect to the usual, more restrictive strongly connectedness assumption -- is that the (Perron) eigenvectors $\qD{i} $'s might have zero elements.
In addition, we require one technical condition.
\begin{assumption}\label{asm:diagonal_con}
	For all $\i$, there is a  matrix $\mat{Q}_i \succ 0$ such that $\sigma_i \coloneqq\|\WD{i} - \1_{N_i} {\qD{i}} ^\top \|_{\mat{Q}_i} <1$,  $[\1_{N_i}^\top \mat Q_i ]_{i_i} = 1$, $\ 1_{N_i}^\top \mat{Q}_i  \WD{i}(\id_{N_i} -\1_{N_i}{\qD{i}}^\top) = \0_{N_i}^\top $,   and either
	\emph{(i)} $\mat Q_i$ is diagonal, or \emph{(ii)} $\Omega_i = \R^{\n{x_i}}$. \hfill $\square$
\end{assumption}
\begin{remark}\label{rem:asm6isweak} 
	 \cref{asm:diagonal_con}\emph{(i)} alone is general enough to comprise all the cases considered in the existing literature (corresponding to the choice $\gD{p} =\g{C}$ in \eqref{eq:standard}):
	\begin{enumerate}[leftmargin = *]
		\item[ \it i.] if $\gD{i}$ is strongly connected with self-loops, then \cref{asm:diagonal_con}\emph{(i)} holds with $\mat{Q}_i=\diag(q_i/[q_i]_{i_i} )$ \cite[Lem.~1]{Bianchi_Directed_CDC2020}; in particular, if  $\WD{i}$ is doubly stochastic, $\mat{Q}_i= \id$. This also means that, if $\g{C}$ is strongly connected, one can always choose the design graphs to satisfy \cref{asm:diagonal_con,asm:row_spanning};
		\item[\it ii.] if  $\gD{i}$ is the directed star graph (namely, there are all and only the  edges from node $i$ to every node in $\neigo{E}{i}$), then \cref{asm:diagonal_con}\emph{(i)} holds with $\mat{Q}_i=  \id $ (and $\sigma_i =0$, $q_i$ with only one nonzero element $[q_i]_{i_i}=1$); note that having this structure for all $\i$ correspond to the classical full-information scenario, as detailed  below. 
	\end{enumerate} 
   Other relevant cases, satisfying \cref{asm:diagonal_con} but never addressed in literature, are discussed after the next result.  \hfill $\square$

\end{remark}
\begin{theorem}\label{th:NE}
	Let \cref{asm:game_convexity,asm:strong_mon,asm:row_spanning,asm:diagonal_con} hold, and let
	\begin{align*}
		\Xi & \coloneqq  \diag  (( \mat{Q}_i \otimes  \id_{\n{x_i}} )_{\i}) \qquad 	\bar \sigma \coloneqq \textstyle  \max_{\i} \{ \sigma_i \}
		\\
		\bar{\theta} & : = \theta  \sqrt{ \textstyle \max_{\i}\{ [ \mat Q_i]_{i_i,i_i} \} / 
		{\eigmin(\Xi)}}
		\\
		\underline \gamma & \coloneqq  \textstyle \sqrt{ 1/\max_{\i}\{ \1^\top \mat{Q}_i \1 \}}, \quad  \bar \gamma\coloneqq \sqrt{ 1/\textstyle \min_{\i} \{ \1^\top \mat{Q}_i \1 \}}
		\\[1em]
	    \mat 	M_\alpha  & \coloneqq
		\begin{bmatrix}
			1 -2\alpha \mu \underline \gamma^2 +\alpha^2\theta^2 \bar \gamma^2
			&
			\bar \sigma (\alpha(\bar\theta+\theta \bar \gamma)+\alpha^2\bar\theta\theta\bar \gamma)
			\\[1em]
			\bar \sigma (\alpha(\bar\theta+\theta \bar \gamma)+\alpha^2\bar\theta\theta\bar \gamma)
			&
			\bar \sigma^2 (1+ 2\alpha\bar \theta +\alpha^2 \bar \theta^2 )
		\end{bmatrix}
	\end{align*}
	Let  $\alpha>0$ be chosen such that
	\begin{align}\label{eq:rhoalphaleq1}
		\rho_\alpha=\eigmax(\mat M_\alpha) <1.
	\end{align}
	Then the sequence $(\hy^k)_{\k}$ generated by \eqref{eq:proxgrad_compact} converges linearly to $\hy^\star \coloneqq \Ebs{} (x^\star )$, where $x^\star$ is the \gls{NE} of the game \eqref{eq:game}: for all $\k$,
	$
		\| \hy^{k+1}  -   \hy ^\star\|^2_{\Xi}  \leq {\rho_\alpha} \|\hy^{k}  -   \hy^\star \|_{\Xi}^2.
	$ \hfill $\square$
\end{theorem}

\smallskip
\begin{proof}
    See \cref{proof:NE}.
\end{proof}

\vspace{1em}
The condition \eqref{eq:rhoalphaleq1} always holds for $\alpha $  small enough (explicit bounds are obtained as in \cite{Bianchi_Timevarying_LCSS2021}). Let us now highlight some of the novelties of \cref{th:NE}:
\begin{enumerate}[leftmargin=*]
	\item[(a)] Consider the standard scenario in \eqref{eq:standard}, where agents store and exchange an estimate of the whole $x$. If $\g{C}$ is strongly connected and $\W{C}$ is doubly stochastic, \cref{th:NE} retrieves exactly \cite[Th.~1]{Bianchi_Timevarying_LCSS2021}. If $\W{C}$ is only row stochastic, \cref{th:NE} improves on the results in \cite{Bianchi_Directed_CDC2020} since \eqref{eq:proxgrad_compact} does not require the knowledge of any Perron eigenvector, but just a small-enough step (as in \cite{Nguyen_Directed_arXiv2023}, which however requires a much stronger monotonicity assumption); this is achieved  by using the weight matrix $\Xi$ in the analysis; 
    \item[(b)] To our knowledge, the only other works that consider partial coupling are  \cite{Salehisadaghiani_Nondoubly_EAI2020,Salehisadaghiani_Graphical_AUT2018}. The authors propose gossip algorithms where, assuming
    a  lower bound for the strongly connected graph $\g{C}$ related to $\g{I}$,
    each agent must only estimate the actions that directly affect its cost.\footnote{This is also achieved in \cref{th:NE} by the choice $\gD{i}=( \neigo{I}{i}, \E{C} \cap (\neigo{I}{i}\times \neigo{I}{i} ))$ if all the resulting $\{\gD{i}\}_{\i}$ are strongly connected, which is a much weaker assumption than \cite[Asm.~6]{Salehisadaghiani_Nondoubly_EAI2020}.}  However, this setup  requires that the  cost of each agent only depends on the actions of its  communication neighbors and neighbors' neighbors \cite[Lem.~3]{Salehisadaghiani_Nondoubly_EAI2020}. By  allowing some  agents to  estimate a larger subset of actions (if needed), \cref{th:NE} avoids this limitation. 
	\item[(c)]  \cref{th:NE} also allows for graphs $\gD{i}$'s that are not strongly connected. For instance, if $\{\gD{i}\}_{\i}$ are all star graphs (i.e., \cref{rem:asm6isweak}.{\it ii}), the action  update in \eqref{eq:proxgrad_compact}  is
	\begin{align}\label{eq:fullinfo_pseudo-gradient}
		x_i^{k+1}=\proj_{\Omega_i} (x_i^k -\alpha \nabla_{x_i} f_i (x^k) )
	\end{align}
    which is the standard pseudo-gradient method for the full-information scenario, where  estimates and true values coincide. In particular, when $\g{E}$ is complete, \cref{th:NE} retrieves  the well-known bound $\alpha <  2\mu / \theta ^2$ (since $\bar \sigma=0$, $\underline \gamma^2 =\bar \gamma^2=\frac{1}{N}$) \cite[Prop.~26.16]{BauschkeCombettes_2017}.
\item[(d)] Another (not strongly-connected) case not addressed before is that of a matrix $ \WD{1}=\left[\begin{smallmatrix}
  1~ & 	  \0_{N_i -1}^\top \\    c~ &  \bar W
\end{smallmatrix}\right]$,  $c\in \R_{\geq 0}^{N_i -1}$, representing a leader-follower protocol (with agent $1$ as the leader for ease of notation). If $\Omega_1 = \R^{\n{x_1}}$ and \cref{asm:row_spanning} holds, it can be checked that \cref{asm:diagonal_con} is verified with ${\qD{1}} = e_1 = \col(1,0,\dots,0)$,  $\mat Q_1 = \left[\begin{smallmatrix}
    1+\1^\top \mat X_{22} \1~ & - 1^\top \mat X_{2,2} \\
    - \mat X_{22} \1_{N_i-1}~ & X_{22}
\end{smallmatrix}\right]$,
where $\mat X= \left[ \begin{smallmatrix}
\mat X_{1,1}~ & \mat X_{1,2} \\ \mat \mat X_{1,2}^\top~ & \mat X_{2,2}
\end{smallmatrix} 
\right]\succ 0$ is any  matrix such that $\| \WD{1}-\1 {\qD{1}}^\top\|_{\mat X}<1$. As a special case, if $[c]_{j_1} = 1$, we have $\hat y_{j,1}^k = x_1^k$, namely agent $j$ can use the real action $x_1$ received by agent $1$  when evaluating its cost. We can also model a scenario in which the \emph{exact} information on $x_1$ propagates over $\gD{1}$ but with some delay, by choosing $\gD{1}$ as a directed tree. 


	\item[(e)] Each graph $\gD{i}$ can be chosen independently. For instance, one variable $x_i$ might be publicly available (choose $\gD{i}$ as a star graph), while other actions can only be reconstructed via consensus. The convergence of these cases would otherwise require ad-hoc analysis (e.g., to determine bounds on the step sizes).
\end{enumerate}


\subsection{GNE seeking in aggregative games}\label{subsec:GNEalgorithms}


Next, we also address the presence of the coupling constraints $Ax \leq  a$. These are  partitioned in $M$  row-blocks, i.e.,
$\mat A=[\mat A_{m,i}]_{\m,\i} $, $a=\col((\sum_{\i} a_{m,i})_{\m})$,
where,  for all $ m \in \M\coloneqq\{1, 2,\dots,M \}$,  $\mat A_{m,i} \in \R^{n_{\lambda_m} \times n_{x_i}}$ and 
$a_{{m},i} \in \R^{n_{\lambda_m}}$ are local data kept by agent $i$. The coupling constraints sparsity pattern is described by the interference graph $\gsup{I}{\lambda} =(\M,\I, \Esup{I}{\lambda}) $, where $(m,i) \in \Esup{I}{\lambda}$ if agent $i$ is involved in  the  constraints block  indexed by $m$; in other terms,
\begin{align}
	(\forall  (m,i) \notin \Esup{I}{\lambda})   \quad  \mat A_{m,i}=\0, \ a_{{m},i}=\0,
\end{align}
and the $m$-th block constraint can be written as $\sum_{i \in \neigosup{I}{m}{\lambda} } \mat A_{m,i}x_i -a_{m,i} = \0$. Correspondingly, we partition the dual variable as $\lambda = \col((\lambda_m) _\m)$, $\lambda_m \in \R^{\n{\lambda_m}}$.

Furthermore, we focus on the particularly relevant class of aggregative games \cite{KoshalNedicShanbhag_Aggregative_OR2016,Parise_Almost_TCNS2020}, where
the agents only  need to reconstruct an aggregation value (usually of dimension independent of $N$, e.g., the average of all the actions \cite{KoshalNedicShanbhag_Aggregative_OR2016}) to evaluate their objective functions.  In particular, the cost coupling  arises via an aggregation mapping $\sigma:\R^{n_x}  \rightarrow \R^\n{\sigma} $,  so that for all $\i$ and for some function $\bar f_i$,
\begin{align} 
f_i(x_i,x_{-i}) =\bar  f_i(x_i,\sigma(x)).
\end{align}
 Let $\sigma$ be partitioned as $\sigma= \col( (\sigma _q)_\q) $, $\Q\coloneqq\{1,2,\dots,Q\}$, and let $\gsup{I}{\sigma} =(\Q,\I,\Esup{I}{\sigma})$ be an interference graph such that, for all $x$ (and with the usual overloading)
\begin{align}
	\sigma_q(x) &= \sigma_q((x_i)_{i\in \neigosup{I}{q}{\sigma} }  )
	\\
	\bar{f_i}(x_i,\sigma(x)) & = \bar{f_i}( x_i, (\sigma_q(x))_{q\in \neigsup{I}{i}{\sigma}}),
\end{align}
namely, $(q,i) \in \Esup{I}{\sigma}$ whenever \emph{either}  $\sigma_q(x)$ explicitly depends on $x_i$, or $\bar{f_i}(x_i,\sigma)$ explicitly depends on $\sigma_q$. We restrict the analysis to \emph{affine} aggregation functions, so that 
\begin{align}
	\sigma_q (x) & \coloneqq \textstyle \sum_{i\in \neigosup{I}{q}{\sigma} } \mat  B_{q,i} x_i +  b_{q,i},
	\\
	\sigma (x) &\hphantom{:}= \col( (\sigma _q)_\q)  = \mat  B x + b,
\end{align}
  $ \mat B_{q,i} \in \R^{\n{\sigma_q} \times \n{x_i}}$, $b_{q,i} \in \R^{\n{\sigma_q}}$ being local data of agent $i$, $\mat B \coloneqq [\mat B_{q,i}]_{q\in \set Q, \i}$, $b = \sum_{\i} \col( (b_{q,i})_{q\in\set Q}) $, with     
  \begin{align}
     (\forall (q,i) \notin \Esup{I}{\sigma}) \quad  \mat B_{q,i} =  \0, \ b_{q,i}=\0.\end{align}

If $M=Q=1$, we recover the standard generalized aggregative games \cite{GadjovPavel_Aggregative_TAC2021,Bianchi_GNEPPP_AUT2022}, where each agent must estimate the whole dual variable $\lambda\in \R^ \n{\lambda}$ and the whole aggregative value $\sigma(x)\in \R^\n{\sigma}$.  Instead, our idea is to leverage the possible problem sparsity by assigning to each agent copies of only some of the components of the dual variable and of the aggregation function, as specified by two\footnote{Here we consider \gls{END} problem for two variables of interest $\lambda$ and $\sigma$. The notation in Section \ref{sec:mathbackground} is recovered by setting $y= \col(\lambda,\sigma(x))$,  with $P=M+Q$. However, for readability, here we treat the two variables separately, as two independent instances of \gls{END}, where the corresponding quantities in \cref{subsec:notation} are distinguished via the superscripts $\lambda$ and $\sigma$ (e.g., $\gDlambda{m}$ and $\gDsigma{q}$ describes how the agents exchange the estimates of $\lambda_m$ and $\sigma_q$, respectively). We also indicate the estimates of $\lambda_m$ and $\sigma_q(x)$ kept by agent $i$  with $\hlambda_{i,m}$ and $\hsigma_{i,q}$ (in place of  $\hy_{i,m}^\lambda$ and $\hy_{i,q}^\sigma$), and analogously for the stacked vectors --- e.g., $\hsigma_q= \col((\hsigma_{i,q})_{i\in \neigosup{E}{q}{\sigma}})$, $\hsigma = \col((\hsigma_{q})_\q)$ and $\hlambda_m= \col((\hlambda_{i,m})_{i\in \neigosup{E}{m}{\lambda}})$, $\hlambda = \col((\hlambda_{m})_\m)$.} estimate graphs $\gsup{E}{\lambda} = (\set M, \set I, \Esup{E}{\lambda}) $, $\gsup{E}{\sigma} = (\set Q, \set I, \Esup{E}{\sigma}) $.  We postpone the discussion of a motivating application to \cref{sec:numerics}.


Let us define
the \emph{extended pseudo-gradient mapping}
\begin{align}
\hspace{-1em}
\Fbstilde(x,\hsigma ) & \coloneqq \col( (\grad{x_i} \bar f_i (x_i,\hsigmat_i)+ \tilde {\mat B}_i^\top \grad{\hsigmat_i } \bar{f}_i  (x_i,\hsigmat_i)) _{\i}),
\end{align}
where  $ \tilde{\mat B}_i \coloneqq \col ((\mat B_{q,i})_{q \in \neigsup{E}{i}{\sigma}})$, and we recall that $\hsigmat_i$ are the estimates of the components of $\sigma$ kept by agent $i$ and (with the customary overloading) $ \bar f_i(x_i,\hsigmat_i) = f_i(x_i,(\hsigma_{i,q})_{q\in \neigsup{E}{i}{\sigma}}) \coloneqq f_i(x_i, (\hsigma_{i,q})_{q\in \neigsup{I}{i}{\sigma}})$. Note that  $\Fbstilde$
coincides with the pseudo-gradient mapping when the estimates are exact and at consensus, i.e., $\Fbstilde(x,\Ebs{}^\sigma  (\sigma(x)))  = \op F(x)$. We study the following distributed iteration:
\begin{subequations}\label{eq:GNE_algo}
\begin{align}
	\label{eq:GNE_algo:x}
	x^{k+1} & =\proj_{\Omega} \Big( 
    \begin{multlined}[t]
    x^k -\beta \big( \alpha \Fbstilde(x^k,\hsigma^k ) + {\bs{\mat B}}^\top  \LDbssup{\sigma} \hsigma^k  \\ +\bs{\mat A}^\top \hlambda^k \big) \Big) 
    \end{multlined}
    \\
	\label{eq:GNE_algo:sigma}
	\hsigma^{k+1} &=\hsigma^{k}-\beta\LDbssup{\sigma} \hsigma^k + { \bs{\mat B}} (x^{k+1}-x^k)
	\\
		\label{eq:GNE_algo:z}
	\hz^{k+1}& = \hz^k+\beta \LDbssup{\lambda} \hlambda^k
	\\
		\label{eq:GNE_algo:lambda}
	\hlambda^{k+1} & = \proj_{\R_{\geq 0}^ \n{\hlambda}} \Big(
    \begin{multlined}[t]
        \hlambda^k- \beta \big(\LDbssup{\lambda} (2\hz^{k+1}-\hz^k ) \\ -\bs{ \mat A}(2x^{k+1}-x^k ) +\bs{a} \big)
     \Big), \end{multlined} \\[-1em] \nonumber
\end{align}
\end{subequations}
where $\alpha,\beta>0$ are step sizes; 
$
  \bs{z} = \col ((\hz_m)_{m\in \set M}) \in \R^{\n\hlambda}$ and $ \bs{z}_{m} = \col((\hz_{i,m})_{i\in\neigosup{E}{m}{\lambda}})$ where, for all $ m \in \neigsup{E}{i}{\lambda}$, $\hz_{i,m}\in \R^{\n{\lambda_m}}$ is an  auxiliary dual variable kept by agent $i$; 
\begin{align*}
\bs{\mat A} &\coloneqq  \mat {P^\lambda} ^ \top \diag( (\tilde{\mat A}_i)_{\i} ), \qquad 
\tilde{\mat A}_i  \coloneqq \col((\mat A_{m,i})_{m\in\neigsup{E}{i}{\sigma}})
\\
\bs{a} &  \coloneqq {\mat P^\lambda} ^\top   \col((\tilde a_i)_{\i}),
\qquad
\tilde a_i  \coloneqq \col ((a_{m,i})_{m \in \neigsup{E}{i}{\lambda}})
\\
\bs{\mat B} & \coloneqq \diag((N_q^\sigma I_{\n{\hsigma_{\!q}}})_\q) \mat \, {\mat P^\sigma} ^ \top \diag( (\tilde{\mat B}_i)_{\i} )
\\
\bs{b} &\coloneqq \diag((N_q^\sigma I_{\n{\hsigma_{\!q}}})_\q) \mat \, {\mat P^\sigma} ^ \top  \col((\tilde b_i)_{\i})
\\
\tilde b_i & \coloneqq \col((b_{q,i})_{q\in \neigsup{E}{i}{\sigma}})
\end{align*}
and we recall that $N_q^\sigma = | \neigosup{E}{q}{\sigma}|$, that with $\tilde{}$ we indicate agent-wise stacked quantities and that  ${\mat {P}^\lambda}$ and ${\mat {P}^\sigma}$ are permutation matrices (i.e., ${\mat {P}^\lambda}^\top
\tilde{\hlambda} = \hlambda$, ${\mat{P}^\sigma}^\top
\tilde{\hlambda} = \hlambda$).
We impose
\begin{align}\label{eq:hsigmainitialization}
    \hsigma^0 = \bs{\mat B} x^0 + \bs{b},
\end{align}
or, agent-wise, $\hsigmat_i^0 = \col( (N^\sigma_q\mat B_{q,i}x_i^0 + N^\sigma_q b_{q,i})_{q\in \neigsup{E}{i}{\sigma}})$; the other variables are initialized arbitrarily. 

The algorithm is based on primal-dual pseudo-gradient dynamics, where the update of each $\hsigma_q$ represents a dynamic tracking of the aggregation function $\sigma_q$, over the graph $\gDsup{q}{\sigma}$. It is inspired by the methods in \cite{Pavel_GNE_TAC2020,GadjovPavel_Aggregative_TAC2021}, and as in these works it is derived as a \gls{FB} method  \cite[§26.5]{BauschkeCombettes_2017}. 

%
%
%
%

\begin{assumption}\label{asm:GNEconnected}
	For each $\m$,  $\gDsup{m}{\lambda}$ is undirected and  connected and $\WDsup{m}{\lambda}$ is symmetric. For each $\q$,   $\gDsup{q}{\sigma}$ is  strongly connected and $\WDsup{q}{\sigma}$ is balanced. \hfill $\square$
\end{assumption}
\begin{theorem}\label{th:GNE}
    Let \cref{asm:game_convexity,asm:strong_mon,asm:GNEconnected}
    hold, and assume that $\Fbstilde(x,\hsigma)$ is $\bar \theta$-Lipschitz continuous. Then, for any small-enough $\alpha>0$ there is a small-enough $\beta >0 $ such that the sequence $(x^k,\hsigma^k,\hz^k,\hy^k)_{\k} $ generated by \eqref{eq:GNE_algo} converges to a point $(x^\star,\Ebs{}^\sigma(\sigma(x^\star)),\hz^\star,\Ebs{}^{\lambda}(\lambda^\star))$, where $(x^\star,\lambda^\star)$ satisfies the \gls{KKT} conditions in \eqref{eq:KKT}, hence $x^\star$ is the \gls{v-GNE} of the game in \eqref{eq:game} (and the estimates of aggregation and dual variables are at consensus). 
    \hfill $\square$
\end{theorem}

\smallskip
\begin{proof}
    See \cref{proof:GNE}, where we also provide explicit bounds for $\alpha$ and $\beta$. 
\end{proof}

To our knowledge, we are the first to consider the partial-coupling in the constraints or aggregation in generalized games. For non-generalized games, an approach is studied in \cite{Eslami:Unicast:TAC:2022}, where the cost of each agent is only affected by some of the components of an aggregation function: nevertheless, this algorithm  requires strong conditions on the (undirected) communication network -- necessary to allow the choice
$\g{I} = \g{E}$, see also \cref{ex:minimal} in \cref{app:tutorial}. Besides avoiding this limitation, we also considered the presence of coupling constraints and a more general aggregation function. 

The efficacy and improved scalability guaranteed by the method in \eqref{eq:GNE_algo} is demonstrated numerically in the next section.


\section{Illustrative example: Unicast rate allocation}\label{sec:numerics}


We study a bandwidth allocation problem with fixed routing \cite{Alpcan:Basar:CongestionControl:CDC:2012,Eslami:Unicast:TAC:2022},  modeled as a \gls{GNE} problem -- see \cref{fig:unicastscheme} for an illustration. Consider an undirected connected communication network $\g{C} =(\set{I},\E{C})$. From each node $i\in\mc{I}$, a user sends data with rate  $x_{i} \in [0,1]$, over a path $\set{L}_i$ -- \ie a sequence of consecutive edges over the graph $\g{C}$. User $i$ aims to choose $x_i$ to minimize the cost function
\begin{align*}
    J_i(x_i,x_{-i}) &\hphantom{:}= J_i(x_i,\{\sigma_l\}_{l\in \set L{i}}) = -U(x_i) + \textstyle \sum_{l \in \set{L}_i} c_l(x_i,\sigma_l)\\
    \sigma_l & \coloneqq \textstyle \sum_{j \mid l \in \set{L}_j} x_j,
\end{align*}
where $U_i$ is an utility function, $\sigma_l$ represents the aggregative rate over the link (\ie edge) $l$, and $c_l$ is a penalty related to link $l$ (\eg quantifying the loss of  service quality due to congestion \cite{Alpcan:Basar:CongestionControl:CDC:2012} or a  tax imposed by a network manager \cite{Eslami:Unicast:TAC:2022}). Furthermore, the capacity  of each link $l$ is bounded by the coupling constraints $\sigma_l \leq a_l $, for some  $a_l >0$. The objective is to seek a \gls{v-GNE} of the resulting generalized game, when the users can only communicate over the graph $\g{C}$.
The capacity of each link can change over time \cite{Alpcan:Basar:CongestionControl:CDC:2012} (e.g., because some external primary users have priority on the bandwidth use), thus the users might have to solve the problem multiple times. Therefore, it is important for the users to  solve the problem as efficiently as possible, to  quickly adapt to the changes. 
\begin{figure}[t] 
\centering
\includegraphics[width=0.85\columnwidth]{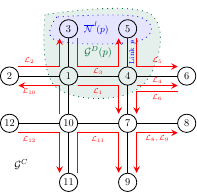}
\caption{Unicast rate allocation game. The main complication is that some users send data over the same link but cannot communicate: for instance, link $(4,5)$ -- labeled $p$ -- is used by users $3$ and $5$ (\ie $\neigo{I}{p} = \{3,5\}$), that are not communication neighbors.
 } \label{fig:unicastscheme}
\end{figure}

Let us relabel the ``active'' edges $\bigcup_{i\in\mc{I}} \{l \in \set L_i \}$ 
(\ie the edges of $\g{C}$ hosting at least one path) as $\{1,2,\dots,P\} \eqqcolon \set P$; further, let us define for all $\p$ and $\i$,
\begin{align*}
     \mat{A}_{p,i} = \mat{B}_{p,i} = \left\{ \begin{aligned}
         1, & \quad  \text{ if $p\in \set L_i$ (user $i$ uses link $p$)}
         \\
         0, & \quad \text{ otherwise},
     \end{aligned}
     \right.
\end{align*}
the so-called \emph{routing} matrix \cite{Alpcan:Basar:CongestionControl:CDC:2012}.
 With these definitions, the problem retrieves an aggregative game of the type  described in \cref{subsec:GNEalgorithms}, with $\set M = \set Q = \set P$. In particular,  only the aggregative variables and constraints relative to the links in $\set L_i$ directly affect agent $i$, \ie $\neigsup{I}{i}{\sigma} = \neigsup{I}{i}{\lambda} = \set L_i$.\footnote{The setup also fits  more general cases where from a node $i$ multiple distinct flows $x_{i,1},\dots,x_{1,\n{x_i}}$ are sent over distinct paths $\set L_{i,1},\dots,\set L_{i,\n{x_i}}$, by either a single user (adjusting the cost function) or by multiple competitive users (considering an augmented communication graph ${\g{C}}^\prime \supset \g{C} $).}

A similar setup was  considered in \cite{Salehisadaghiani_Graphical_AUT2018} (as a non-aggregative game, which is inefficient) and \cite{Eslami:Unicast:TAC:2022}; however, both works cannot deal with the capacity coupling constraints, and further assume that the graphs $\{\g{C}|_{\neigo{I}{p}}\}_{\p}$ are  connected, which is for example not the case in \cref{fig:unicastscheme} (see also \cref{ex:minimal} in \cref{app:tutorial}). Hence, the methods in \cite{Salehisadaghiani_Graphical_AUT2018,Eslami:Unicast:TAC:2022} cannot be applied. 


To seek a \gls{v-GNE}, we implement algorithm in \eqref{eq:GNE_algo}.
In particular, we fix  $\WDsup{p}{\sigma}=\WDsup{p}{\lambda} \eqqcolon \WD{p}$ for all $\p$. We are interested in comparing the  the performance of the algorithm for two different choices of the design graphs:
\begin{itemize}[leftmargin =*]
    \item \emph{Standard}: $\WD{p}$'s are chosen as in \eqref{eq:standard}, \ie  each agent estimates the whole dual variable and aggregative value: the problem sparsity is ignored;
    \item \emph{Customized}:  $\WD{p}$'s are chosen to minimize the per-iteration communication cost,  by solving \cref{prob} with: ``\emph{(iii)} For all $\p$, $\gD{p}$ is connected and  has minimal number of edges''.
\end{itemize}
We note that, for the customized algorithm, the design of each graph $\WD{p}$ corresponds to solving a Steiner tree problem as in \cref{fig:0}. This is achieved in a distributed manner, and it is further only done once even if the link capacities change after some time (as long as the routing is fixed).

\begin{figure}
\includegraphics[width=0.95\columnwidth]{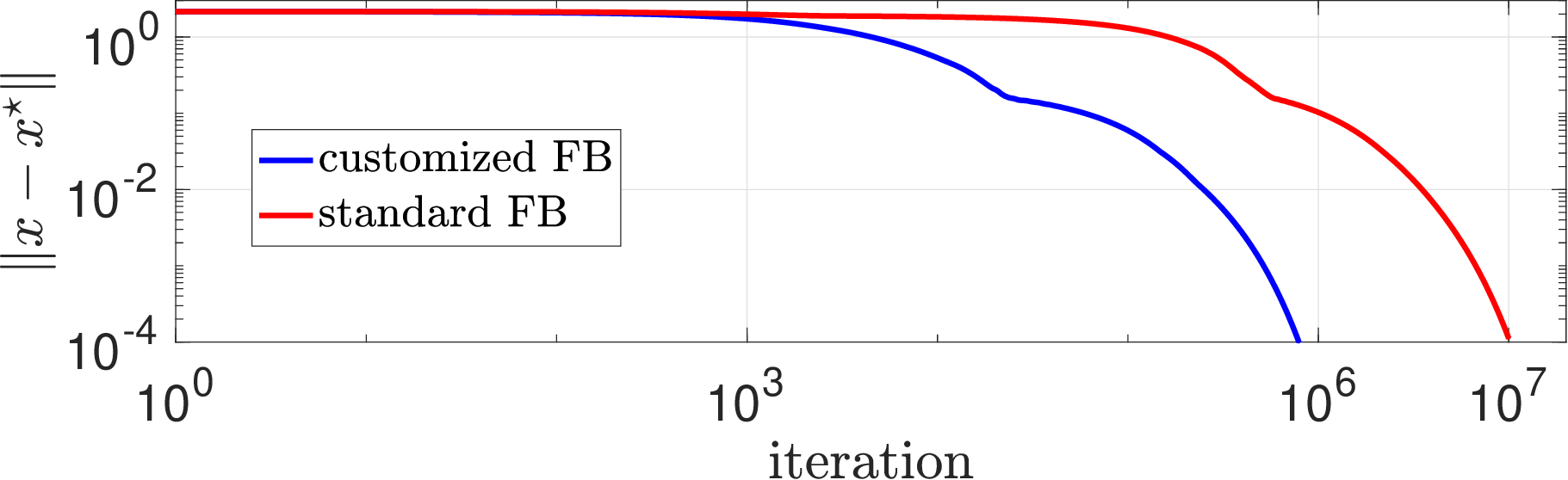}
\\[1em]
\includegraphics[width=0.95\columnwidth]{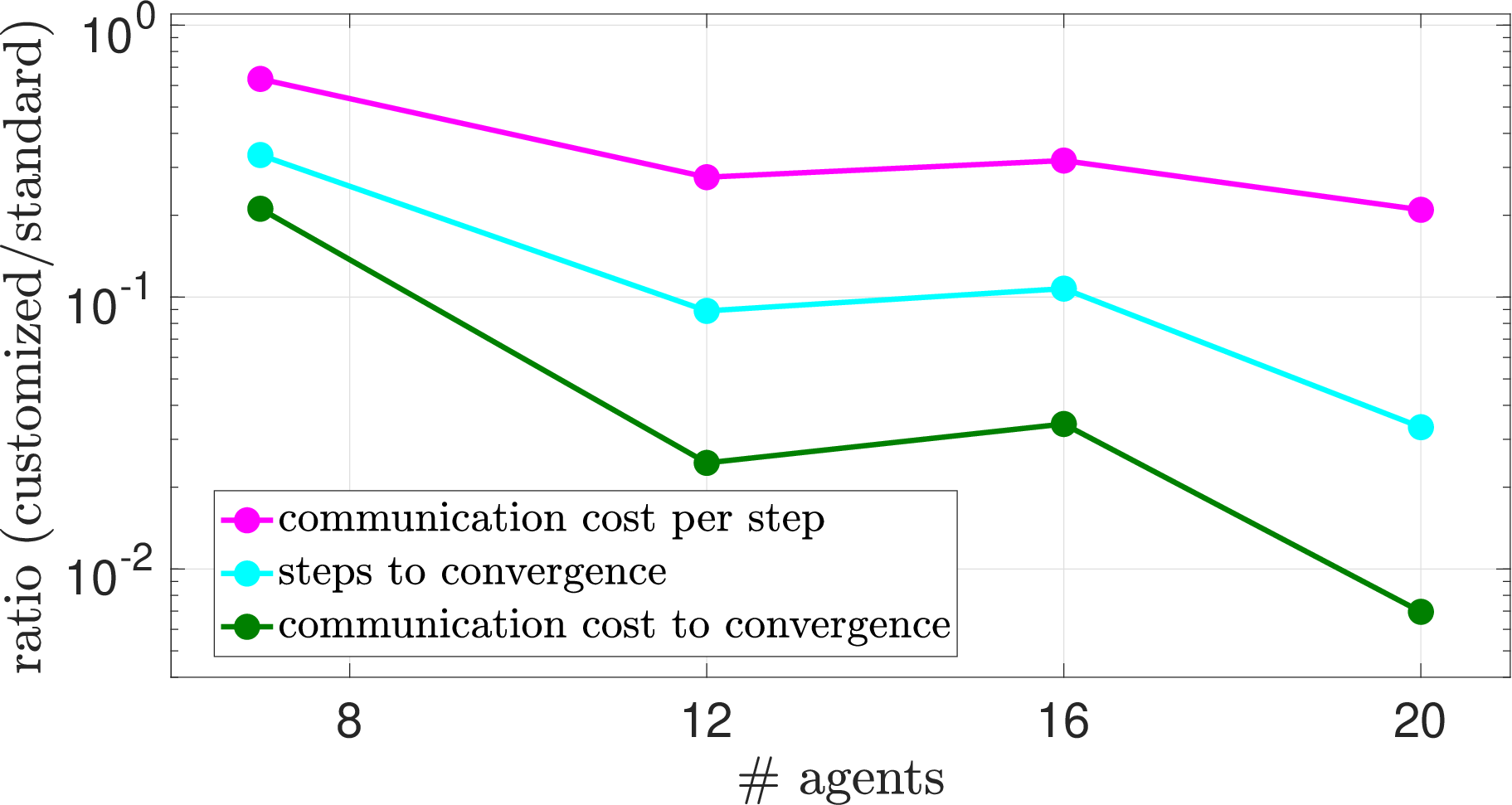}
\caption{Unicast rate allocation via algorithm \eqref{eq:GNE_algo}, for the scheme in \cref{fig:unicastscheme} (top) and for different randomly generated networks, with maximal path length $4$ and stopping criterion $\|x-x^\star\| \leq 10^{-2}$  (bottom).} \label{fig:unicastsimulation}
\end{figure}

We choose $U_i(x_i) = 10*\log(x_i+1)$, $c_l (x_i,\sigma_l) = \psi_l \frac{x_i}{1+e^{-\sigma_l}}$ with $\psi_l$ sampled randomly uniformly in the interval $[0,1]$
(\cref{asm:strong_mon} is satisfied in the feasible set, invariant for \eqref{eq:GNE_algo}), Metropolis-Hastings weights for every graph, $\alpha = 0.1$, $\beta = 10^{-3}$. 
The results are illustrated in \cref{fig:unicastsimulation}. In our first experiment, we consider the scenario in \cref{fig:unicastscheme}, with $I = P =12$. The customized algorithm converges to the unique \gls{v-GNE} $x^\star$ over $10$ times faster than the standard version; the communication burden at each step is also reduced, as the mean size of the (aggregative/dual) estimates kept and transmitted by each agent is $2.6$, instead of $P = 12$. Since the latter quantity grows with the problem dimension, we expect the gap between the two methods to increase for larger networks. Simulations with $I = 7,12,16,20$ ($P= 7,16,18,26)$ confirm this intuition: for the case $I = 20$, the customized algorithm saves $99 \%$ of the communication cost (where sending one scalar value to \emph{one} neighbor on $\g{C}$ costs $1$, in unicast fashion). 
Finally, for the selected parameters and $I \geq 24$, the standard algorithm fails to converge in our simulations, while the customized algorithm converges (at least) up to $I =50$ -- suggesting tolerance to larger step sizes, with upper bounds less affected by the problem dimension. 

In conclusion, while requiring some initial design effort, the customized method can result in substantial efficiency improvement, especially if the \gls{v-GNE} problem is solved multiple 
times due to time-varying cost parameters or link capacities.

\section{Conclusion}\label{sec:extension}

We  presented \gls{END}, a graph-theoretic language of consensus in distributed iterations. 
Our framework
allows for unprecedented flexibility in the assignment and exchange of estimates among the agents. In fact, \gls{END} algorithms can be tailored to exploit the  sparsity of  specific problem instances, improving scalability and reducing communication and memory bottlenecks, without requiring case-by-case convergence analysis.

We have exploited  \gls{END} to design algorithms for \gls{GNE}  seeking under partial-decision information, that improve on known ones in terms of both theoretical guarantees and numerical efficiency. Yet, \gls{END} can be applied to virtually any distributed decision problem, e.g., common fixed point computation, consensus optimization, multi-cluster games, aggregative optimization \cite{Li_AggregativeOpt_TCNS2022}.

 Future work should focus on computationally efficient and distributed methods to perform the allocation of the estimates; in particular, it would be highly valuable to dynamically assign the estimates online, thus avoiding the need for any a priori design. Another interesting research direction is to combine  \gls{END} algorithms with other communication-reduction techniques, such as data sparsification or compression \cite{Koloskova:Stich:Jaggi:Compressed:ICML:2019}, which can further enhance efficiency.

\appendix
\subsection{Examples of the \gls{END} design phase} \label{app:tutorial}

In this section, we present some examples of \cref{prob} and discuss choices for the design and estimate graphs.

\subsubsection{Minimal memory allocation} \label{sec:minimalmemory} We consider the problem of minimizing, for each $\P$,  the number of copies of $y_p$, provided that the  conditions in \cref{prob:1,prob:2} are satisfied and the graphs $\{\gD{p}\}_{\p}$ enjoy some connectivity properties. In particular, consider \cref{prob} with
\begin{itemize}
\item[\emph{(iii)}] for each $\p$, $\gD{p}$ is rooted at $r_p \in \mc{I}$  [respectively, $\gD{p}$ is strongly connected]; the number $|\neigo{E}{p}|$ of nodes in $\gD{p}$  is minimal (provided that all the other specifications are satisfied).
\end{itemize}

If the problem is feasible, then by definition a solution is given by choosing each
$\gD{p}$ 
as a solution of 
an Unweighted Directed Steiner tree problem [of a Strongly Connected Subgraph problem] \cite{Charikar:Directed:Steiner:1999} (similarly to \cref{fig:0}).
A sufficient condition for the existence of a solution is that $\g{C}$ is strongly connected. 

Note that the solution  is generally not unique, specially because the specification is only given in terms of nodes. In fact, given any optimal choice $\mc{G}'$ for $\gD{p}$, any graph $\mc{G}'' $ such that 
$\mc{G}' \subseteq \mc{G} '' \subseteq \g{C} $ is 
is also a solution for \cref{prob}. In simple terms, we can add edges to $\mc{G}'$, an extra degree of freedom that can be employed to improve connectivity or robustness to link failure (possibly at the cost of extra communication). One can also impose a different connectedness/efficiency specification on each graph $\gD{p}$ (see \cref{fig:0}).

\subsubsection{Fair allocation and bandwidth constraints}
Instead of minimizing  the overall dimension of $\hy$, it can be convenient to promote allocations where the memory occupation (or the communication requirements) are partitioned equally among the agents: for example, assuming that undirected connected graphs are required, this could be achieved by  sequentially designing the graphs $\gD{p}$'s as solutions of a Steiner tree problem (see \cref{sec:background}), but sequentially penalizing unbalanced allocations by opportunely choosing the edge weights for each $p$. Bandwidth constraints can be addressed similarly, to avoid overloading some channels of the communication network $\g{C}$.

\subsubsection{Designing the communication graph} 

In this paper, we consider the graph $\g{C}$ as given,
which is  natural  for ad hoc networks  or when relying on existing infrastructures. Yet, other works \cite{Salehisadaghiani_Graphical_AUT2018,Salehisadaghiani_Nondoubly_EAI2020,Eslami:Unicast:TAC:2022}  assume that the communication network can be freely designed. In the \gls{END} framework, this case is addressed by formally assuming that  $\g{C}$ is complete; then the graphs $\{ \gD{p}\}_{\p}$ can be chosen to fulfill some specifications (e.g., minimize the number of active edges in $\g{C}$ -- which determines the physical channels/edges  actually needed). 

\subsubsection{Straightforward designs} 
\label{subsubsec:straightforward}
In some cases, the problem structure immediately suggests an optimal estimate allocation.
\begin{example}[Partitioned optimization] \label{ex:partitioned}
Motivated by distributed estimation and  resource allocation applications, the works \cite{Notarnicola_Partitioned_TCNS2018,Erseghe_ADMM_SPL2012,Todescato_Partition_AUT2020} solve optimization problems of the form
	\begin{equation} \label{eq:localdo}
	\min_{y_i \in\R^{\n{y_i}}, i\in\mc{I}} \  \textstyle\sum_{\i} f_i(y_i, (y_j)_{j \in\neig{C}{i}} ),
\end{equation}
where the cost $f_i$ of agent $i$ depends on its local action $y_i$ and on the actions of its neighbors over the undirected communication network $\g{C}$: this is a special case of \cref{ex:do1}, with $\P=\I$,  $\neigo{I}{i} = \neig{C}{i} \cup \{i\}$.
Consider \cref{prob}, with ``\emph{(iii)}~$\forall \i$, $\g{D}_i$ is connected;  $\g{D}_i$ has minimum number of nodes (provided that all other specifications are met)''. A solution is to fix each $\gD{i}$ as the undirected star graph centered in $i$  with vertices $\V{D}_i =\neigo{I}{i}$: 
then agent $i$ keeps all and only proxies of the actions that affect its cost. In fact, this is the solution employed in  \cite{Notarnicola_Partitioned_TCNS2018,Erseghe_ADMM_SPL2012,Todescato_Partition_AUT2020}. \hfill $\square$
	\end{example}\begin{example}[$\g{E} = \g{I}$]\label{ex:minimal}
With the goal of minimizing the overall memory allocation, consider the choice 
\begin{align}\label{eq:minima}
 \gD{p} = \g{C}|_{\neig{I}{p}}, \qquad \forall \p.
 \end{align}
 In this case,  $\g{E} = \g{I}$, \ie each agent only estimates the minimum amount of variables needed for local computation. Yet, this is only a viable option  if the resulting graphs $\{\gD{i}\}_{\i}$
ensure the desired connectedness properties in \cref{prob:3},  which usually not verified (see $\gD{1}$ in \cref{fig:0}), but holds in some particular cases (\eg \cref{ex:partitioned}, $\gD{2}$ in \cref{fig:0}; see also \cite[Asm.~5]{Salehisadaghiani_Graphical_AUT2018}, \cite[Asm.~6]{Salehisadaghiani_Nondoubly_EAI2020} for sufficient conditions in the context of \gls{NE} seeking). \hfill $\square$
\end{example}
\begin{example}
    Another example where \eqref{eq:minima} is a viable option was studied in \cite{Salehisadaghiani_Nondoubly_EAI2020}, for \gls{NE} problems arising in social networks, where the cost of agent $i$ is $f_i(y_i,(y_j)_{j \in \neig{C}{i}},(y_\ell)_{l \in \neig{C}{j},j \in\neig{C}{i}})$, \ie the cost of each agent depends on its own action $y_i$ and on the actions of its (in-)neighbors and  neighbors' neighbors. \hfill $\square$
\end{example}

\subsection{Proofs}

\subsubsection{Proof of \cref{th:NE}} \label{proof:NE}
We study convergence of \eqref{eq:proxgrad_compact} in the space weighted by
$\Xi \succ 0$.
  Let
$\hy^\star\coloneqq\Ebs{}({x^\star})$, where $x^\star$ is the \gls{NE} of \eqref{eq:game}. Our proof is based on the following lemma.
\begin{lemma}\label{lem:NEcrux}
	Let  $\Fmc(\hy)\coloneqq\WDbs \hy - \alpha  \Rmc ^\top \Fbs (  \WDbs \hy)$. Then, for any $\hy\in\R^\n{\hy}$, it holds that
	\begin{equation*}
		\|  \Fmc(\hy) - \Fmc(   \hystar) \|_\Xi \leq \sqrt{\rho_\alpha} \| \hy -\hystar\|_\Xi. \QEDopenhereeqn
	\end{equation*}
\end{lemma}
\vspace{1em}
\begin{proof}
	Let $\hy= \hypar+\hyperp$, where $\hypar \coloneqq \diag ( ( (\1_{N_i} {\qD{i}}^\top) \otimes \id_{\n{x_i}}  )_{\i}) \hy \in \Ebs{}
	$, and thus $\hypar =  \Ebs{}(\yhatpar)$ for some $\yhatpar \in \R^{\n{y}}$.
	Let $\hyhat\coloneqq \WDbs \hy =  \hyhatpar+\hyhatperp$, where $\hyhatperp \coloneqq \WDbs \hyperp = \WDbs \diag( ((\id-\1_{N_i}  {\qD{i}}^\top) \otimes \id_{\n{x_i}})_{\i}) \hy $ and  we used that $ \WDbs \hypar= \hypar = \Ebs{}({\yhatpar})$ (by row stochasticity). By \cref{asm:diagonal_con}, we  have  $ \qD{i} \1_{N_i}^\top \mat{Q}_i \WD{i} (\id-\1_{N_i}  {\qD{i}}^\top) =\0$, and hence $ \langle \hyhatpar, \hyhatperp \rangle_{\Xi} =0$. Therefore 
	%
	\begin{align*}
		&  \hphantom {{} = {}}	\| (\hyhat - \alpha \Rmc^\top  \Fbs (\hyhat )) - (\hystar - \alpha \Rmc ^\top \Fbs (\hystar)) \|_\Xi^2
		\\
		&  = \| \hyhatpar - \hystar \|_\Xi^2 + \| \hyhatperp \|_{\Xi}^2
		\\
		& \hphantom{{} = {}}  +
		\alpha^2 \| \Rmc^\top(\Fbs(\hyhat) -\Fbs(\hyhatpar)+\Fbs(\hyhatpar)-\Fbs(\hystar))\|_{\Xi}^2
		\\
		&  \hphantom {{} = {}} - 2\alpha  \langle  \hyhatperp, \Rmc^\top(\Fbs(\hyhat) -\Fbs(\hystar) \rangle_\Xi
		\\
		&  \hphantom {{} = {}} - 2\alpha \langle   \hyhatpar - \hystar , \Rmc^\top(\Fbs(\hyhat) -\Fbs(\hyhatpar) \rangle_\Xi
		\\
		&   \hphantom {{} = {}} - 2\alpha \langle   \hyhatpar - \hystar , \Rmc^\top(\Fbs(\hyhatpar) -\Fbs(\hystar) \rangle_\Xi
		\\
		\nonumber
		& \leq
		\| \hyhatpar - \hystar \|_\Xi^2 + \| \hyhatperp \|_{\Xi}^2  +\alpha^2\textstyle
		(  \bar{\theta}\| \hyhatperp \|_{\Xi} + \theta \bar{\gamma} \| \hyhatpar - \hyhatpar\|_\Xi)^2
		\\
		\nonumber
		&  \hphantom {{} = {}} +2\alpha
		\bar{\theta}
		\|  \hyhatperp \|_{\Xi} (\| \hyhatpar - \hystar \|_\Xi+\| \hyhatperp \|_{\Xi})
		\\
		\numberthis
		\label{eq:NE_usefulstep}
		&  \hphantom {{} = {}}  + 2\alpha
		\theta \bar\gamma\| \hyhatpar - \hystar \|_\Xi \| \hyhatperp \|_{\Xi}
		- 2\alpha
		\mu \underline \gamma ^2
		\| \hyhatpar - \hystar \|_\Xi^2
	\end{align*}
	where the last inequality follows by the Cauchy--Schwartz inequality and using that $\| \Rmc^\top v\|^2_{\Xi} \leq \max_{\i} \{ [\mat Q_i]_{i_i,i_i}\} \| v \|^2 $ for all $v\in \R^\n{x}$, because $\Rmc \Xi \Rmc^\top =\diag(([\mat Q_i]_{i_i,i_i} \id)_{\i}) $;  that $\Fbs$ is $\theta$-Lipschitz continuous if $\op F$ is 
	(see \cite[Lem.~1]{Bianchi_Timevarying_LCSS2021}); finally, that $\|\Fbs({\hyhatpar})-\Fbs(\hystar)\|^2= \| \op F(\yhatpar)- \op F(x^\star)\|^2 \leq \theta^2 \| \yhatpar-x^\star\| = \theta^2\| \hyhatpar - \hystar\|_{\diag((\mat{Q}_i\otimes \id /(\1^\top \mat Q_i \1))_{\i})}$
	(the last equality due to $(\hyhatpar-\hystar) \in \Ebs{}$)
	and similarly that   $\langle   \hyhatpar - \hystar , \Rmc^\top(\Fbs(\hyhatpar) -\Fbs(\hystar))\rangle_\Xi= \langle   \yhatpar - x^\star ,  \op F(\yhatpar) - \op F(x^\star) \rangle  \geq \mu \|\yhatpar-x^\star \|^2 = \mu   \|\hyhatpar-\hystar \|^2_{\diag((\mat{Q}_i\otimes \id /(\1^\top\mat Q_i \1))_{\i})}$, 
	(by using the normalization $[ \1^\top \mat Q_i]_{i_i}= 1$ in \cref{asm:diagonal_con}).
	  In addition, by \cref{asm:row_spanning}, we have $ \|\hyhatperp \|_{\Xi} = \|\WDbs  \hyperp \|_\Xi \leq \bar{\sigma} \| \hyperp \|_\Xi  $; together with \eqref{eq:NE_usefulstep}, this yields 
	\begin{align*}
		& \hphantom{ {} \leq {} } \|  \Fmc(\hy) - \Fmc(   \hystar)) \|_\Xi
		\leq \begin{bmatrix}
			\| \hypar - \hystar \|_\Xi
			\\
			\| \hyperp \|_{\Xi}
		\end{bmatrix}^\top
		\textrm{M}_\alpha
		\begin{bmatrix}
			\| \hypar - \hystar \|_\Xi
			\\
			\| \hyperp \|_{\Xi}
		\end{bmatrix}
		\\
		& \leq \eigmax (\textrm{M}_\alpha) (\| \hypar - \hystar \|_\Xi ^2 +\| \hyperp \|_{\Xi}^2)
		\\
		& =\eigmax (\textrm{M}_\alpha) \| \hy- \hystar \|_\Xi ^2. \vspace{-1em} 
		\\[-3em]
	\end{align*}
\end{proof}
To conclude the proof of \cref{th:NE}, we note that \eqref{eq:proxgrad_compact} is equivalently written as $\hy^{k+1} = \proj^\Xi _{\bs{\Omega}} \left (  \WDbs \hy^k - \alpha  \Rmc ^\top \Fbs (   \WDbs \hy^k) \right)$, where $\proj_{\bs{\Omega}}^\Xi$ is the projection in $\mc{H}_{\Xi}$ (i.e., $\proj_{\bs{\Omega}}^\Xi(\bs x) = \argmin_{\bs y \in \bs{\Omega}} \| \bs{x} - \bs{y} \|_{\Xi} $) . In fact, $\proj_{\bs{\Omega}} =\proj_{\bs{\Omega}}^\Xi$ block-wise under either  \cref{asm:diagonal_con}\emph{(i)} (due to block diagonality of  $\mat Q_i$ and the rectangular structure of $\bs{\Omega}$) or \cref{asm:diagonal_con}\emph{(ii)} (trivially). Moreover, $\hystar$ is a fixed point for \eqref{eq:proxgrad_compact}.  By nonexpansiveness of the projection operator \cite[Prop.~12.28]{BauschkeCombettes_2017}, we can finally write
\begin{align*}
	\| \hy^{k+1} - \hystar \|_{\Xi } & = \| \prox^\Xi _{\bs{g}} (\Fmc( \hy )) -  \prox^\Xi _{\bs{g}} (\Fmc(\hystar))\|_{\Xi}
	\\
	& \leq
	\| \Fmc( \hy )- \Fmc(\hystar)\|_{\Xi},
\end{align*}
and the conclusion follows by \cref{lem:NEcrux}.
\hfill $\blacksquare$

\subsubsection{Proof of \cref{th:GNE} }
\label{proof:GNE}
We can rewrite \eqref{eq:GNE_algo} with the change of variable $\hs \coloneqq \hsigma - \bs{ \mat B} x -\bs{b} $ by replacing $\hsigma ^k$ with $\hs^k+ \bs {\mat B} x^ k + \bs{b}$ in \eqref{eq:GNE_algo:x}, and  \eqref{eq:GNE_algo:sigma} with  
\begin{align}	\label{eq:GNE_algo:s}
        \hs^{k+1} =  \hs^k  - \beta \LDbssup{\sigma} (\hs^k + \bs{\mat B} x^k+\bs{b} ), \quad \hs^0 = \0.
    \end{align} 
    Let us define $\homega \coloneqq\col (x,\hs,\hz,\hlambda)$,
\allowdisplaybreaks
\begin{align}
\label{eq:opA}
 \mathfrak{A}(\homega) & \! \coloneqq \!
\underbrace
{\begin{bmatrix}
    \alpha \Fbstilde (x,\hsigma)+\bs{\mat B}^\top \LDbssup{\sigma } \hsigma
    \\
    \LDbssup{\sigma} \hsigma 
    \\
    \0 
    \\
    \bs{b}
\end{bmatrix} }_{ \coloneqq \mathfrak{A}_1}
\!
+
\!
\underbrace{\begin{bmatrix}
    \bs{\mat A}^ \top  \hlambda
    \\ 
    \0 
    \\ 
    - \LDbssup{\lambda} \hlambda
    \\
    \LDbssup{\lambda}\hz - \bs {\mat A } x 
\end{bmatrix}}_{\coloneqq \mathfrak A_2}
\! + \! 
\underbrace{\begin{bmatrix}
        \nc_{\Omega} \\
        \0 \\ \0 \\ \nc_{\R^{\n{\hlambda}}_{\geq 0}}
    \end{bmatrix}}_{\coloneqq \mathfrak A_3} 
    \\
    \label{eq:Phimatrix}
     \Phi  & \coloneqq \begin{bmatrix}
       \beta^{-1} \id  & \0 & \0
       & -\bs{\mat A} ^\top 
       \\
       \0 & \beta^{-1} \id  & \0 & \0
       \\
       \0 & \0 & \beta^{-1} \id  & \LDbssup{\lambda}
       \\
       - \bs{\mat A} & \0 & \LDbssup{\lambda} & \beta^{-1} \id 
    \end{bmatrix}
\end{align}
where $\hsigma$   here is just  a shorthand notation for $\hsigma = \hs + \bs{ \mat B} x+\bs{b} $. We assume that $0<\beta< 1/ \| \bs{\mat A}\|_\infty +   \| \LDbssup{\lambda} \|_{\infty} $, 
so that $\Phi \succ 0$; denote $\delta \coloneqq \uplambda_{\min}(\Phi)$. The proof is based on the following auxiliary results. We recall the definition of $\Piparallel{}$ (the projection matrix into the consensus subspace) and $\Ebs{}(\cdot)$ (see \cref{eq:Cconsensus}) in \cref{subsec:notation}, and that with the superscript $\sigma$ or $\lambda$ we distinguish the spaces of different dimension related to the variables $\bs{\sigma}$ and $\bs{\lambda}$. 
\begin{lemma}[Invariance]\label{lem:GNE_invariance}
    For all $\k$, $\Piparallel{\sigma} \hs^k = \0$.
    \hfill $\square$
\end{lemma}
\begin{proof}
    Via induction and by \eqref{eq:GNE_algo:s}, since  $ \Piparallel{\sigma} \LDbssup{\sigma} = \0. $
\end{proof}

\begin{lemma}[Algorithm derivation]\label{lem:GNE_derivation}
The iteration in \eqref{eq:GNE_algo}, with \eqref{eq:GNE_algo:sigma} replaced by \eqref{eq:GNE_algo:s}, can be written as 
\begin{align}\label{eq:FBcompact}
    \mathfrak A_2 (\homega^{k+1}) +     \mathfrak A_3 (\homega^{k+1})  +\Phi (\homega^{k+1} - \homega^{k}) \ni   \! -  \mathfrak A_1 (\homega^k),
\end{align}
with $\mathfrak A_1$, $\mathfrak A_2$, $\mathfrak A_3$ as in \eqref{eq:opA}. \hfill $\square$
\end{lemma}
\begin{proof}
  The iteration in \eqref{eq:GNE_algo:x},\eqref{eq:GNE_algo:s},\eqref{eq:GNE_algo:z} \eqref{eq:GNE_algo:lambda} is retrieved 
   expanding the terms in \eqref{eq:FBcompact} and canceling identical terms, by noting that $\mathfrak A_3$ is the normal cone of the set $ \mc{S} = \Omega \times \R^{\n{\hsigma} } \times \R^ \n{\bs{\hlambda}} \times \R^\n{\hlambda}_{\geq 0} $ and hence  $\homega^{k+1} + \mathfrak A_3(\homega^{k+1}) \ni \homega'$ is equivalent to $\homega^{k+1} = \proj_{\mc{S}} (\homega')$ for any $\homega'$.
 \end{proof}

\begin{lemma}[Fixed points]\label{lem:GNE_zeros}
  The fixed points of the iteration in \eqref{eq:FBcompact} coincide with $\zer(\mathfrak A)$. The set $\zer(\mathfrak{A}) \cap  \Sigma \coloneqq \{\homega  \mid  \Piparallel{\sigma} \hs = \0\}$ is nonempty. Moreover, for any   $\homega^\star = (x^\star, \hs^\star, \hz^\star, \hlambda ^\star) \in \zer(\mathfrak{A}) \cap \Sigma$, we have that $\hsigma^\star \coloneqq \hs^\star + \bs{\mat B} x^\star +\bs{b} = \Ebs{}^{\sigma}(\sigma(x^\star))$, $\hlambda^\star = \Ebs{}^{\lambda}(\lambda^\star)$, where $(x^\star, \lambda^\star)$ solve \eqref{eq:KKT}, hence $x^\star$ is the \gls{v-GNE} of the game in \eqref{eq:game}. \hfill $\square$
 \end{lemma}
\begin{proof}
    By definition of inverse operator,  \eqref{eq:FBcompact} is equivalent to $\homega^{k+1} = (\Id+ \Phi^{-1}(\mathfrak{A}_2+\mathfrak A_3)) ^{-1}(\Id - \Phi^{-1} \mathfrak A_3) (\homega^k )$, i.e., 
    the \gls{FB} algorithm  \cite[§26.5]{BauschkeCombettes_2017} applied to the operator $\Phi^{-1}\mathfrak{A}$; its fixed points are the zeros of $\Phi^{-1}\mathfrak{A}$ \cite[Prop. 26.1(iv)]{BauschkeCombettes_2017}, and clearly $\zer({\Phi^{-1}\mathfrak{A}}) = \zer({\mathfrak{A}})$. Now consider any $\homega^\star \in \zer{\mathfrak{A}} \cap \Sigma$. Note that $ \hsigma^\star = \Piparallel{\sigma} \hsigma^\star  =  \Piparallel{\sigma} \hs^\star + \Piparallel{\sigma} \bs{\mat B} x^\star + \Piparallel{\sigma} \bs{b} = 
    \Ebs{}^\sigma (\sigma(x^\star))$, where the first equality follows by the second row in \eqref{eq:opA} and \cref{asm:GNEconnected} (and \cref{lem:rooted}), and the second by definition of $\bs{\mat B}$, $\bs{b}$ and $\Piparallel{\sigma}$. By the third row in \eqref{eq:opA}, $\hlambda^\star \in \Ebs{}^\lambda$; in turn, the first row retrieves the first \gls{KKT} condition in \eqref{eq:KKTa}. The second condition in \eqref{eq:KKTb} is obtained by left-multiplying the last row in \eqref{eq:opA} by $\Piparallel{\lambda}$. Conversely, it can be shown  as in \cite[Lem.~10]{Bianchi_GNEPPP_AUT2022} that, for any $(x^\star,\lambda^\star)$ solving \eqref{eq:KKT} (at least one such pair exists by \cref{asm:strong_mon}), there exists $\hz^\star \in \R^{\n{\hlambda}}$ such that $(x^\star,\hs^\star \coloneqq\Ebs{}^{\sigma}(\sigma(x^\star))-\bs{\mat B} x^\star- \bs b, \hz^\star, \Ebs{}^\lambda{\lambda^\star}) \in \zer({\mathfrak A})  $; therefore $\zer(\mathfrak A) \cap \Sigma \neq \varnothing$, since $\Piparallel{\sigma}\hs^\star=\0$. 
\end{proof}

\begin{lemma}[Monotonicity properties]\label{lem:GNE_monotonicity}
The operator $\mathfrak{A_2}+\mathfrak{A_3}$ is maximally monotone \cite[Def.~20.20]{BauschkeCombettes_2017}. Let $\bar{\uplambda} \coloneqq \uplambda_2(\frac{\LDsup{\sigma}+\LDsup{\sigma}}{2})$; then, for any $0<\alpha < \frac{4\mu  \bar \uplambda}{\bar\theta^2}$, the operator $\mathfrak{A_1}$ is $\eta_{\alpha}$-restricted cocoercive, for some $\eta_{\alpha}>0$ depending on $\alpha$: for all $\homega \in \Sigma$ and all $\homega^\star \in \zer(\mathfrak A) \cap \Sigma$, $\langle \mathfrak A_1(\homega)-\mathfrak{A_1}(\homega^\star), \homega- \homega^\star \rangle \geq  \eta_{\alpha} \|\mathfrak A_1(\homega)-\mathfrak{A_1}(\homega^\star) \| ^2$. \hfill $\square$
\end{lemma}
\begin{proof}
$\mathfrak A_1$ is maximally monotone because it is a skew-symmetric linear operator \cite[Ex~20.35]{BauschkeCombettes_2017}, the normal cone $\mathfrak{A}_2$ is maximally monotone by \cite[Ex.~20.26]{BauschkeCombettes_2017}; then, since $\mathfrak A_1$ has full domain, $\mathfrak A_1+ \mathfrak A_2$ is maximally monotone by \cite[Th.~25.2]{BauschkeCombettes_2017}. 
For the second statement, for all $\homega \in \Sigma$, we have $\langle \homega-\homega^\star,\mathfrak A_1(\homega)-\mathfrak A_1(\homega^\star)\rangle = \alpha\langle x-x^\star, \Fbstilde(x,\hsigma) - \Fbstilde(x^\star,\hsigma^\star)\rangle+ \langle \hsigma-\hsigma^\star, \LDbssup{\sigma}(\hsigma-\hsigma^\star)\rangle $, and the result follows identically to \cite[Lemma~5]{GadjovPavel_Aggregative_TAC2021} 
by using  \cref{lem:connected}, that $\Fbstilde$ is $\bar\theta$-Lipschitz, that $\Fbstilde (x, \Piparallel{\sigma} \hsigma) = F(x)$ and $F(x)$ is $\mu$-strongly monotone (we refer to \cite{GadjovPavel_Aggregative_TAC2021} for the expression of $\eta_\alpha$).
\end{proof}

\cref{lem:GNE_derivation} recasts \eqref{eq:GNE_algo} as a preconditioned \gls{FB} algorithm \cite{Pavel_GNE_TAC2020}, applied to the operators $\mathfrak A_1$, $\mathfrak A_2+\mathfrak A_3$, with preconditioning matrix $\Phi \succ 0$. In turn, \cref{lem:GNE_monotonicity}  ensures the conditions on the operators $\mathfrak{A}_1$ and $\mathfrak{A}_2+\mathfrak{A}_3$ that guarantee convergence of the preconditioned \gls{FB} method to a fixed point in $\zer(\mathfrak A) \cap \Sigma$, provided that the stepsize $\beta$ is chosen such that $0<\beta < 2\eta_\alpha \delta$ (this can be proven by showing the decrease of the Lyapunov function $\|\homega^k- \homega^\star\|_{\Phi}$, with analysis restricted to the invariant subspace $\Sigma$, as in \cite[Th.~1]{GadjovPavel_Aggregative_TAC2021}\cite[Th.~2]{Pavel_GNE_TAC2020}; the  argument is standard and omitted here due to space limitations). Finally, \cref{lem:GNE_invariance,lem:GNE_zeros} characterize such a fixed point as per statement.
\hfill $\blacksquare$

\bibliographystyle{IEEEtran}
\bibliography{library}

\begin{IEEEbiography}[{\includegraphics[width=1.0
    in,height=1.35in,clip,keepaspectratio]{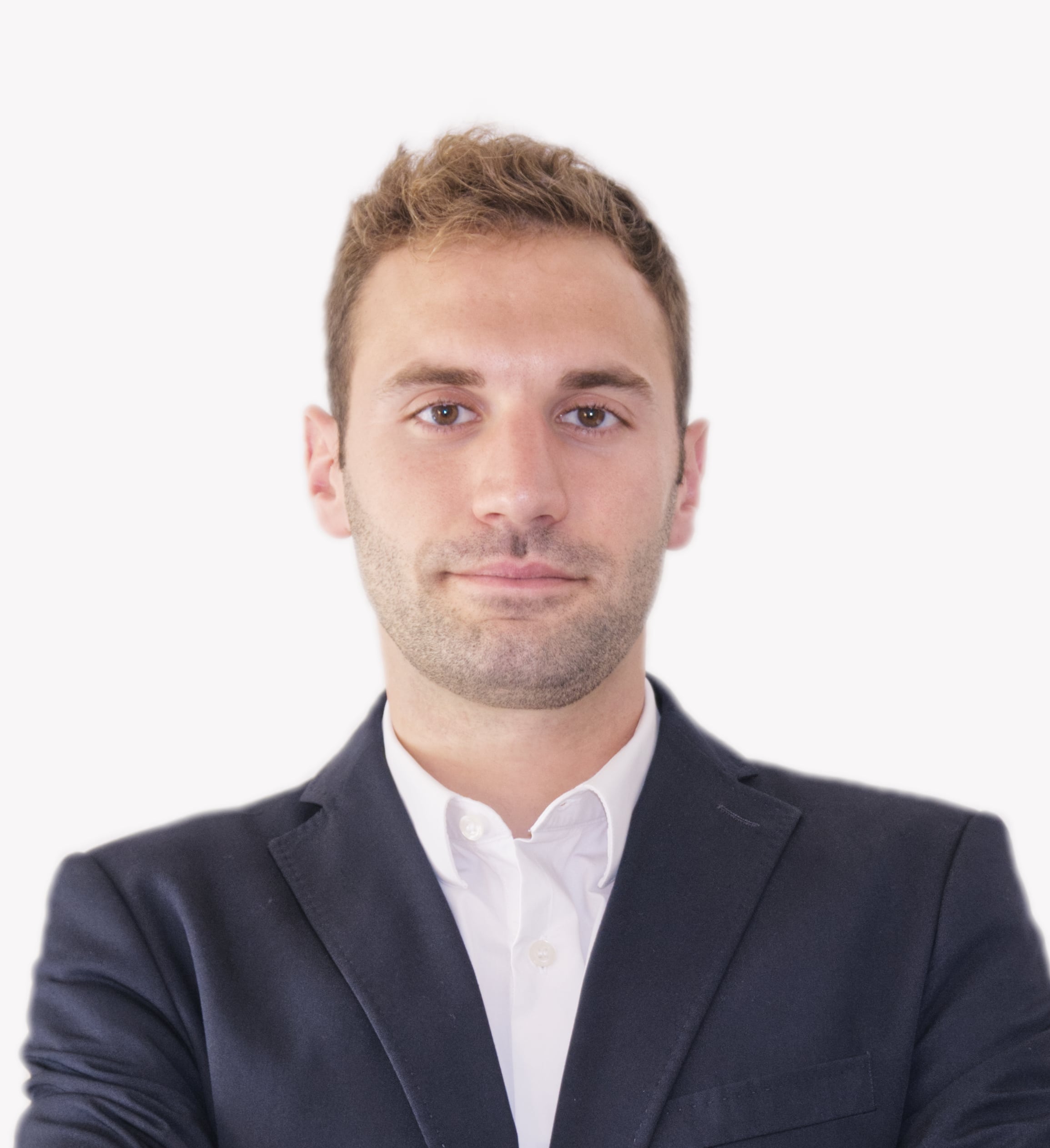}}]{Mattia Bianchi}{\space} is a postdoctoral researcher at the Authomatic Control Laboratory (IfA), ETH Zürich, Switzerland. He received the Bachelor’s degree in Information Engineering and the Master’s degree in Control Systems Engineering, both from University of L’Aquila (IT), in 2016 and 2018, respectively. He received the Ph.D. degree in Automatic Control from Delft University of Technology (NL) in 2023. In 2018 he visited the Control Systems Technology group, TU Eindhoven (NL). In 2021-2022 he visited the Mechanical and Aerospace Engineering Department, University of California San Diego (USA). His research interests include game theory and operator theory to solve decision and control problems for complex systems of systems, in the presence of network coupling and uncertainty. He is recipient of the 2021 Roberto Tempo Best Paper Award at the IEEE Conference on Decision and Control.
\end{IEEEbiography}

\begin{IEEEbiography}[{\includegraphics[width=1.0
    in,height=1.35in,clip,keepaspectratio]{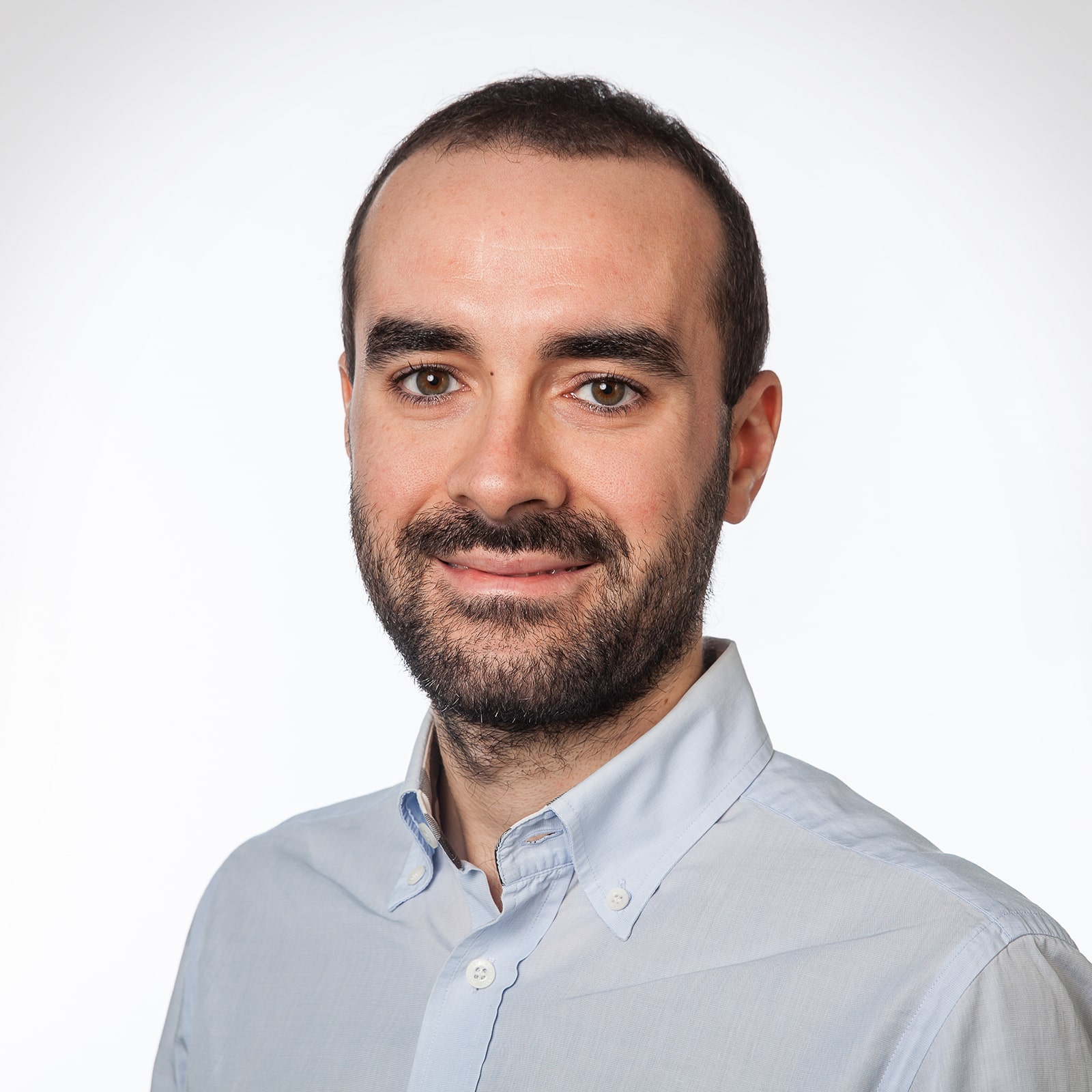}}]{Sergio Grammatico}{\space} 
 is an Associate Professor at the Delft Center for Systems and Control, TU Delft, The Netherlands. Born in 1987, he received the Bachelor’s degree in Computer Engineering, the Master’s degree in Automatic Control Engineering, and the Ph.D. degree in Automatic Control, all from the University of Pisa, Italy, in February 2008, October 2009, and March 2013 respectively. He also received a Master’s degree in Engineering Science from the Sant’Anna School of Advanced Studies, the Italian Superior Graduate School (Grande École) for Applied
Sciences, in November 2011. In 2013–2015, he was a postdoc researcher in the Automatic Control Laboratory, ETH Zurich, Switzerland. In 2015–2018, he was an Assistant Professor in the Department of Electrical Engineering, Control Systems, TU Eindhoven. He was awarded a 2005 F. Severi B.Sc. Scholarship by the Italian High-Mathematics National Institute, and a 2008 M.Sc. Fellowship by the Sant’Anna School of Advanced Studies.  He is a recipient of the 2021 Roberto Tempo Best Paper Award at the IEEE Conference on Decision and Control. He is currently an Associate Editor of the IEEE Trans. on Automatic Control (2018–present) and of Automatica (2020–present).
\end{IEEEbiography}

\end{document}